\documentclass[reqno,11pt,letterpaper]{amsart}
\usepackage[mathscr]{eucal}
\usepackage{amsmath}	
\usepackage{amssymb}	
\usepackage{amsthm}
\usepackage{booktabs}
\usepackage{amsfonts}
\usepackage{latexsym}
\usepackage{mathrsfs}
\usepackage{tabularx}
\usepackage{pict2e}
\usepackage{nicematrix}
\usepackage{float}   
\usepackage{graphicx}
\usepackage[all]{xy}
\usepackage{tikz-cd}
\usepackage{tikz}
\usetikzlibrary{matrix,arrows}
\usepackage[colorlinks=true, citecolor=blue, urlcolor=blue, linkcolor=blue]{hyperref}
            
\theoremstyle{plain} 
\newtheorem{thm}{Theorem}[section]
\newtheorem{lem}[thm]{Lemma} 
\newtheorem{prop}[thm]{Proposition} 
\newtheorem{cor}[thm]{Corollary}

\newtheorem*{mainconj}{Conjecture A}
\newtheorem{quest}[thm]{Question}
\newtheorem*{mainquest}{Question B}
\newtheorem*{mainquest2}{Question C}
\theoremstyle{definition} 
\newtheorem{defn}[thm]{Definition}
\newtheorem{setup}[thm]{Set-up}
\newtheorem{rem}[thm]{Remark} 
\newtheorem{ex}[thm]{Example}

\makeatletter
\@namedef{subjclassname@2020}{%
  \textup{2020} Mathematics Subject Classification}
\makeatother

\makeatletter
\def\@tocline#1#2#3#4#5#6#7{\relax
  \ifnum #1>\c@tocdepth % then omit
  \else
    \par \addpenalty\@secpenalty\addvspace{#2}%
    \begingroup \hyphenpenalty\@M
    \@ifempty{#4}{%
      \@tempdima\csname r@tocindent\number#1\endcsname\relax
    }{%
      \@tempdima#4\relax
    }%
    \parindent\z@ \leftskip#3\relax \advance\leftskip\@tempdima\relax
    \rightskip\@pnumwidth plus4em \parfillskip-\@pnumwidth
    #5\leavevmode\hskip-\@tempdima
      \ifcase #1
       \or\or \hskip 1em \or \hskip 2em \else \hskip 3em \fi%
      #6\nobreak\relax
    \dotfill\hbox to\@pnumwidth{\@tocpagenum{#7}}\par
    \nobreak
    \endgroup
  \fi}
\makeatother
\newcommand{\NN}{\mathbb{N}}
\newcommand{\ZN}{\mathbb{Z}}

\newcommand{\RN}{\mathbb{R}}
\newcommand{\CN}{\mathbb{C}}
\newcommand{\PB}{\mathbb{P}}
\newcommand{\CO}{\mathcal{O}}
\newcommand{\CA}{\mathcal{A}}
\newcommand{\CB}{\mathcal{B}}
\newcommand{\RA}{\mathscr{A}}
\newcommand{\RB}{\mathscr{B}}
\newcommand{\RT}{\mathscr{T}}
\newcommand{\DC}{\mathrm{D}^{{\rm b}}}
\newcommand{\Hom}{\mathrm{Hom}}
\newcommand{\Ext}{\mathrm{Ext}}
\newcommand{\ext}{\mathrm{ext}}

\newcommand{\RHom}{\mathrm{RHom}}
\newcommand{\RCH}{R\mathcal{H}om}
\newcommand{\CExt}{\mathcal{E}xt}
\newcommand{\HH}{\mathrm{HH}}
\newcommand{\NHH}{\mathrm{NHH}}
\newcommand{\ph}{\mathrm{ph}}
\newcommand{\phac}{\mathrm{ph_{ac}}}
\newcommand{\Pic}{\mathrm{Pic}}

\allowdisplaybreaks
\numberwithin{equation}{section}
\begin{document}

\title{Echoes of phantoms on rational surfaces}

\author{Shihao Ma}
\address{Center for Applied Mathematics and KL-AAGDM, Tianjin University, Weijin Road 92, Tianjin 300072, P. R. China}%
\email{shma@tju.edu.cn}%

\author{Yirui Xiong}
\address{School of Sciences, Southwest Petroleum University, Chengdu 610500, P. R. China}
\email{yiruimee@gmail.com}%

\author{Song Yang}
\address{Center for Applied Mathematics and KL-AAGDM, Tianjin University, Weijin Road 92, Tianjin 300072, P. R. China}%
\email{syangmath@tju.edu.cn}%

\begin{abstract}
We construct the first countably infinite family of new universal phantom categories on a smooth rational surface, namely the blow-up of the complex projective plane at ten points in general position. 
Moreover, these phantom categories are pairwise non-equivalent, are not equivalent to any previously known phantom category on a smooth rational surface, and arise as the orthogonal complements of non-full exceptional collections of line bundles of maximal length.
As an application, we show that all of these phantom categories admit bounded $t$-structures.
\end{abstract}

\date{\today}


\maketitle

\setcounter{tocdepth}{1}
\tableofcontents

%=====================================================================

\section{Introduction}
Let $X$ be a smooth complex projective variety. 
We denote by $\DC(X)$ the bounded derived category of coherent sheaves on $X$.
A non-trivial admissible subcategory $\mathscr{A}$ of $\DC(X)$ is called a {\it quasi-phantom} if the Hochschild homology of $\mathscr{A}$ vanishes and its Grothendieck group is finite; moreover, if its Grothendieck group vanishes, it is called a {\it phantom}.
The examples of quasi-phantom subcategories were first constructed on some surfaces of general type by B\"{o}hning--Graf von Bothmer--Sosna \cite{BGvBS13} on the classical Godeaux surface, by Alexeev--Orlov \cite{AO13} on Burniat surfaces, by  Galkin--Shinder \cite{GS13} on Beauville surface, and by Galkin--Katzarkov--Mellit--Shinder \cite{GKMS15} on Keum's fake projective planes.
The first examples of phantom categories were constructed on products of surfaces of general type with quasi-phantom subcategories by Gorchinskiy--Orlov \cite{GO13} and on determinantal Barlow surfaces by B\"{o}hning--Graf von Bothmer--Katzarkov--Sosna \cite{BGvBKS15}. 
Later, phantom categories exist on some Dolgachev surfaces constructed by Cho--Lee \cite{CL18} and Karzhemanov--Katzarkov \cite{KK23}. 
For the existence of phantom categories on smooth rational surfaces,
it was proved by Pirozhkov \cite{Pir23} that every del Pezzo surface has no phantom categories. 
The first example of phantom categories on smooth rational surfaces was constructed by Krah \cite{Kra24}. 
In \cite{BK25}, Borisov--Kemboi proved that the blow-ups of the complex projective plane $\PB^{2}$ at points lying in general position on a smooth cubic curve do not contain phantom categories.
Recently, some new phantom categories were constructed on the blow-up of $\PB^{2}$ at $11$ points in general position \cite{KKL+26,MXY25} and on the blow-up of the second Hirzebruch surface $\mathbf{F}_{2}$ at $9$ points in general position \cite{KKL+26}.
In \cite{MXY25}, the authors also proved that every smooth projective surface with an effective smooth anti-canonical divisor has no phantoms, e.g. weak del Pezzo surfaces.

At present, all known examples of phantom categories on smooth rational surfaces arise on the blow-up of $\PB^{2}$ at $10$ and $11$ points in general position and on the blow-up of $\mathbf{F}_{2}$ at $9$ points in general position.
It remains an open problem whether one can construct further phantom categories on smooth rational surfaces beyond the currently known examples.
As a matter of fact, it was predicted by Kemboi et al. in \cite[Conjecture 4.11]{KKL+26} that the blow-up of the Hirzebruch surface $\mathbf{F}_{n}$ at $6+\max\{3,n\}$ points in general position has a phantom subcategory $\mathcal{C}_{n}$ orthogonal to an exceptional collection of line bundles of maximal length and $\mathcal{C}_{n} \ncong \mathcal{C}_{m}$ ($n\neq m$). 
If true, this would yield countably infinitely many pairwise non-equivalent phantom categories on smooth rational surfaces. 
In loc. cit., this conjecture has been proved for $n=2$. Since ${\rm Bl}_{2\, {\rm pts}}\, \PB^{2} \cong {\rm Bl}_{1\, {\rm pt}}\, \mathbf{F}_{0}$ and ${\rm Bl}_{1\,{\rm pt}}\, \PB^{2} \cong \mathbf{F}_{1}$, Krah's phantom yields the existence part of this conjecture for $n=0,1$.
However, the case $\mathbf{F}_{n}$ ($n\geq 3$) remains widely open; moreover, even for the case of $\mathbf{F}_{0}$, $\mathbf{F}_{1}$ and $\mathbf{F}_{2}$, the possible existence of new phantoms is still open.
This suggests the following natural conjectural expectation:

\begin{mainconj}\label{conjmain}
There exists a smooth rational surface containing countably infinitely many pairwise non-equivalent phantom categories.
\end{mainconj}

The purpose of this paper is to confirm this expectation by constructing countably infinitely many new phantoms on the blow-up of $\PB^{2}$ at ten points in general position. 
The main idea is to combine mutations with Krah's approach, providing a new perspective on the construction of phantom categories.
Let $X$ be the blow-up of $\PB^{2}$ at $10$ points $p_{1},\cdots,p_{10}$ in general position.
Let $H$ be the pullback to $X$ of the hyperplane class on $\mathbb{P}^{2}$, $E_{i}$ be the exceptional divisors over the point $p_{i}$, $1\leq i\leq 10$, and $K_{X}=-3H+\sum_{i=1}^{10} E_{i}$ be the canonical divisor. 
Consider the involution of the Picard group of $X$,
\begin{equation}\label{involution-Pic}
\begin{array}{cccl}
\iota:& \mathrm{Pic}(X) & \longrightarrow  & \mathrm{Pic}(X)\\
&D&\longmapsto &
-D-2 (D\ldotp K_{X}) K_{X}.
\end{array}
\end{equation}
Then, the involution \eqref{involution-Pic} is linear, $\iota(K_{X})=K_{X}$ and $\chi(D)=\chi(\iota(D))$.
In \cite{Kra24}, Krah applied the involution  \eqref{involution-Pic} to the standard full exceptional collection of line bundles
\begin{equation}\label{standard-FEC-10pts}
\{\CO_{X}, \CO_{X}(E_{1}), \cdots, \CO_{X}(E_{10}), \CO_{X}(H), \CO_{X}(2H)\}
\end{equation}
and obtained a non-full exceptional collection, providing the first example of a phantom category on a smooth rational surface. 
Mutations are a fundamental tool for constructing new full exceptional collections. 
Among them, one of the most basic mutations is the mutation by Serre functor.  
Starting from \eqref{standard-FEC-10pts}, taking mutation by Serre functor and then applying the involution \eqref{involution-Pic}, the resulting phantom category is equivalent to Krah's phantom category.
This motivates the following:

\begin{mainquest}\label{main-motivation}
Does every full exceptional collection of line bundles that is mutation-equivalent to \eqref{standard-FEC-10pts} necessarily give rise to a phantom category?
If so, are the phantom categories orthogonal to the resulting non-full exceptional collections mutually equivalent?
\end{mainquest}

For every integer $a\geq 0$, we take the full exceptional collection
\begin{align}\label{FEC-10points}
 &\{
\CO_{X},\CO_{X}(H-E_{1}-E_{2}),\CO_{X}(E_{3}),\cdots,\CO_{X}(E_{10}), \CO_{X}(H-E_{1}), \nonumber \\ 
 & \;\;\; \CO_{X}((a+1)H-aE_{1}-E_{2}),\CO_{X}((a+2)H-(a+1)E_{1}-E_{2})
\}.     
\end{align}
We note that, by Remark \ref{Ori-FEC-mut-equ-standard}, the full exceptional collection \eqref{FEC-10points} is mutation-equivalent to the full exceptional collection \eqref{standard-FEC-10pts}.
Then, for each integer $a\geq 0$, applying the involution \eqref{involution-Pic} to \eqref{FEC-10points}, 
we obtain the following divisors:
$$
\begin{array}{r@{\hspace{3pt}}lr@{\hspace{3pt}}lcc}
 A:=&2K_{X}-H+E_{1}+E_{2}, & B_{j}:=&2K_{X}-E_{j},  \\
G:=&4K_{X}-H+E_{1},
& 
F_{a}:=&4(a+1)K_{X}-(a+1)H+aE_{1}+E_{2},
\end{array}
$$
where $3\leq j\leq 10$.

The main result of the present paper is stated as follows:

\begin{thm}\label{mainthm} 
Let $X$ be the blow-up of $\PB^{2}$ at $10$ points in general position.
Then, for every integer $a \geq 0$, the sequence 
$$
\{\CO_{X},\CO_{X}(A),\CO_{X}(B_{3}),\cdots,\CO_{X}(B_{10}),\CO_{X}(G),\CO_{X}(F_{a}),\CO_{X}(F_{a+1})\}
$$
is a non-full exceptional collection of line bundles of maximal length. 
In particular, its right orthogonal complement $\RA_{X}^{(a)}$ is a universal phantom category.
\end{thm}

The phantom categories in Theorem \ref{mainthm} corresponding to different integers $a$ are pairwise non-equivalent and are not equivalent to any previously known phantom categories on smooth rational surfaces.
More precisely, we have:

\begin{thm}\label{distinct-phantoms-thm}
For any distinct integers $a,b\geq 0$, the phantom categories $\RA_{X}^{(a)}$ and $\RA_{X}^{(b)}$ are 
neither equivalent to each other nor to any of the phantom categories constructed
in \cite{Kra24,KKL+26,MXY25}.
\end{thm}

On the one hand, Theorem \ref{distinct-phantoms-thm} provides a negative answer to the final part of Question {\bf B}.
Combining with the results in \cite{Kra24} and \cite{KKL+26}, Theorem \ref{distinct-phantoms-thm} also implies that \cite[Conjecture 4.11]{KKL+26} holds for $n\leq 2$. 
On the other hand, a natural question arises as follows, which we hope to address in \cite{MXY26}.

\begin{mainquest2}
Is there a blow-up of $\PB^{2}$ at $10$ points in a special configuration that admits countably infinitely many pairwise non-equivalent phantom categories?  
\end{mainquest2}

Moreover, since the discovery of the first examples of phantom categories, it has remained a well-known open problem whether a phantom category can admit a bounded $t$-structure (\cite[\S 5, Question 2]{Sos20}). More generally, given a smooth projective variety $X$ and an admissible subcategory $\mathscr{C}\subset \DC(X)$, one may ask whether $\mathscr{C}$ admits a bounded $t$-structure (\cite[Question 1.1]{KLP26}).
Recently, Kuznetsov--Liu--Perry \cite{KLP26} introduced a method for inducing $t$-structures of a triangulated category with a $t$-structure on its semi-orthogonal components.
Using this method, they showed in \cite[Theorem 6.12]{KLP26} that almost all previously known examples of phantoms and quasi-phantoms admit bounded $t$-structures. 
They also mentioned that perhaps the most interesting application of their results is to phantom categories.
As an application, we establish the following.

\begin{thm}\label{phantom-bounded-t-structure}
Let $\mathscr{A}_{X}^{(a)}$ be a phantom category in Theorem \ref{mainthm}, where $a\geq 0$.
Then $\mathscr{A}_{X}^{(a)}$ has a bounded $t$-structure. 
\end{thm}

Finally, in Example \ref{10pts-object-in-heart}, we also construct some explicit objects in the heart of bounded 
$t$-structures on each phantom category $\mathscr{A}_{X}^{(a)}$, as in \cite{KLP26}.

\subsection*{Strategy of the proof}
The basic idea of the proof is elegantly simple, and it has been adopted in \cite{Kra24,KKL+26,MXY25}.
Let $X$ be the blow-up of $\PB^{2}$ at $10$ points in general position. 
The first step is to construct a full exceptional collection of line bundles
\begin{equation}\label{ori-FEC-line}
\DC(X)= \langle \CO_{X}(D_{1}), \cdots, \CO_{X}(D_{13}) \rangle. 
\end{equation}
The {\it central challenge}, then, is to identify a suitable such full exceptional collection.
For instance, we note that $X$ can be viewed as the blow-up of $\PB^{1}\times \PB^{1}$ at $9$ points. 
Based on the standard full exceptional collections of line bundles on $\PB^{1}\times \PB^{1}$, by Orlov's blow-up formula \cite{Orl93}, we obtain the desired full exceptional collection of line bundles on $X$.

Next, the {\it key technique} of Krah \cite{Kra24} is to apply an involution $\iota: \Pic(X)\to \Pic(X)$ to a promising \eqref{ori-FEC-line} that preserves the intersection product and the canonical divisor $K_{X}$.
The resulting sequence of line bundles 
\begin{equation}\label{isometry-ori-FEC-line}
\{\CO_{X}(\iota(D_{1})), \cdots, \CO_{X}(\iota(D_{13}))\}  
\end{equation}
is a numerically exceptional collection, i.e. $\chi(\CO_{X}(D_{i}-D_{j}))=0$ for any $1\leq i<j\leq 13$. 
One of the {\it main difficulties} is to establish the non-effectivity of certain divisors.
Previous works \cite{Kra24,KKL+26,MXY25} rely heavily on the Segre--Harbourne--Gimigliano--Hirschowitz (SHGH) conjecture for divisors with small multiplicities, as established in \cite{CM11,DJ07}. In contrast, the multiplicities arising in our setting may be arbitrarily large, a regime where the SHGH conjecture remains widely open. We therefore adopt Dumnicki's diagram-cutting method \cite{Dum07}  to identify divisors with no global sections (see Proposition \ref{10pts-NoSect-case1} and Proposition \ref{10pts-NoSect-case2}).

Finally, to study the orthogonal complement of the constructed exceptional collection, 
we mainly use Kuznetsov's notions of height and normal Hochschild cohomology \cite{Kuz15}.  

\subsection*{Organization}
In Section \ref{Prelim-sect}, we briefly review  semi-orthogonal decompositions, phantom categories, Hochschild cohomology and normal Hochschild cohomology of admissible subcategories, height and pseudoheight of exceptional collections, and Dumnicki's diagram-cutting method.
In Section \ref{desired-FEC}, we recall the construction of the full exceptional collections of line bundles on the blow-up of the complex projective plane $\PB^{2}$ at points in general position that will be used throughout the paper.
Section \ref{10pts-pf-mianthm1} is devoted to the proof of Theorem \ref{mainthm}. 
In Section \ref{Compar-phantoms}, we prove Theorem \ref{distinct-phantoms-thm} by distinguishing the various phantom categories arising in Theorem \ref{mainthm} via their Hochschild cohomology groups.
In Section \ref{appl-t-struture}, we prove Theorem \ref{phantom-bounded-t-structure}, namely, all phantom categories constructed in Theorem \ref{mainthm} admit bounded $t$-structures (Theorem \ref{bd-t-str-ourphant}), and we construct some explicit objects lying in the corresponding hearts.
In Section \ref{Final-remarks}, we propose possible ways to construct further phantom categories via mutations and Krah's approach.
Finally, Appendix \ref{technique-prop2} contains the proof of a technical result. 

\subsection*{Notation and conventions}
We work over the complex number field $\CN$.
A variety is an integral separated scheme of finite type over $\CN$.
For a variety $X$, we use $\DC(X)$ to denote the bounded derived category of coherent sheaves on $X$.
For $E,F\in \DC(X)$, we denote $\ext^{k}(E,F):=\dim \Ext^{k}(E,F)$ and $\hom(E,F):=\dim \Hom(E,F)$, and set 
$$
\RHom(E,F):=\bigoplus_{k\in \ZN} \Hom(E,F[k])[-k].
$$
For a divisor $D$ on $X$, we set $h^{k}(D):=\dim H^{k}(X,\CO_{X}(D))$, $k\in\ZN$, and $\chi(D):=\chi(\CO_{X}(D))$ the Euler characteristic of $\CO_{X}(D)$.

\subsection*{Acknowledgments}
We would like to express our deep gratitude to Professor Xiaojun Chen for his constant encouragement and support.
We thank Professors Tom Bridgeland, Alexander Perry and Evgeny Shinder for their valuable comments on the initial version of this paper.
Finally, we thank the School of Mathematics at Sichuan University and the Tianyuan Mathematical Center in Southwest China for hosting our research visit in the summer of 2026.
This work is partially supported by the National Natural Science Foundation of China (No. 12171351 and No. 12501051), and by Sichuan Science and Technology Program (No. 2025ZNSFSC0800).

%===================================================================

\section{Preliminaries}\label{Prelim-sect}

In this section, we briefly review semi-orthogonal decompositions, phantom categories, Hochschild cohomology and normal Hochschild cohomology of admissible subcategories, height and pseudoheight of exceptional collections, and Dumnicki's diagram-cutting method.

\subsection{Semi-orthogonal decompositions and phantom categories}

Let $X$ be a smooth complex projective variety and $\CA\subset \DC(X)$ be a full triangulated subcategory. 
The {\it left} and {\it right orthogonal complements} of $\CA$
are respectively defined by
$$
{}^{\perp}\CA:=\{E \in \DC(X) \mid \Hom(E,A) = 0 \text{ for all } A \in \CA\}
$$
and
$$
\CA^{\perp}:=\{E \in \DC(X) \mid \Hom(A,E)=0 \text{ for all } A \in \CA\}.
$$

\begin{defn}
We say that a full triangulated subcategory $\mathcal{A} \subset \DC(X)$ is {\it admissible} if the inclusion functor $\imath:\mathcal{A} \hookrightarrow \DC(X)$ has both right adjoint $\imath^{!}$ and left adjoint $\imath^{\ast}$. 
\end{defn}

\begin{rem}
If $\CA\subset \DC(X)$ is an admissible subcategory, then both ${}^{\perp}\CA$ and $\CA^{\perp}$ are admissible subcategories.    
\end{rem}
 
One of the most important examples of admissible subcategories is generated by an exceptional collection.

\begin{defn}
An object $ A\in \DC(X)$ is called {\it exceptional} if 
$$
\RHom(A,A)=\CN[0].
$$
An ordered sequence of exceptional objects $\{A_{1}, \cdots, A_{l}\}$ is called an {\it exceptional collection of length $l$} if 
$$
\RHom(A_{j},A_{i})=0
$$ 
for any $1\leq i<j\leq l$.
\end{defn}

Let $\{A_{1}, \cdots, A_{l}\}$ be an exceptional collection.
We use $\langle A_{1}, \cdots, A_{l}\rangle$ to denote the minimal full triangulated subcategory of $\DC(X)$ containing all objects $A_{i}$, $1\leq i\leq l$. Then, each triangulated subcategory $\langle A_{i} \rangle$ is equivalent to $\DC({\rm Spec} \,\CN)$.

\begin{defn}
We say that an exceptional collection $\{A_{1}, \cdots, A_{l}\}$ is {\it full} if $\DC(X)=\langle A_{1}, \cdots, A_{l}\rangle$.    
\end{defn}

\begin{ex}
(1) On $\PB^{2}$, there is a standard full exceptional collection of line bundles
$$
\DC(\PB^{2})=\langle\CO,\CO(H),\CO(2H)\rangle,
$$
where $H$ is the hyperplane class on $\PB^{2}$.

(2) On $\PB^{1}\times \PB^{1}$, for every integer $a\geq 0$, 
there is a standard full exceptional collection of line bundles
$$
\DC(\PB^{1}\times \PB^{1})=\langle \CO,\CO(P),\CO(Q+aP),\CO(Q+(a+1)P) \rangle,
$$
where $P$ and $Q$ are the two ruling classes on $\PB^{1}\times \PB^{1}$. 
\end{ex}

It is of importance to notice that every full exceptional collection is a special semi-orthogonal decomposition. 

\begin{defn}
An ordered sequence of full triangulated subcategories $\{\CA_{1},\cdots,\CA_{l}\}$ of $\DC(X)$ is called a {\it semi-orthogonal decomposition} of $\DC(X)$ if the following conditions hold:
\begin{enumerate}
\item[(1)] for all $A_{i}\in\CA_{i}$, $A_{j} \in \CA_{j}$, one has $\Hom(A_{i},A_{j})=0$ if $j<i$;
\item[(2)] for any object $T\in \DC(X)$, there exists a chain of morphisms 
$$
\xymatrix@C=0.5cm{
0=T_{l} \ar[r] & T_{l-1} \ar[r] & \cdots \ar[r] & T_{1} \ar[r] & T_{0}=T
}
$$
such that the cone $\mathrm{Cone}(T_{i}\rightarrow T_{i-1})\in \CA_{i}$ for all $1\leq i\leq l$.
\end{enumerate}
Such a semi-orthogonal decomposition is denoted by 
$$
\DC(X)=\langle \CA_{1},\CA_{2},\cdots,\CA_{n} \rangle.
$$
\end{defn}

\begin{ex}
If $\CA\subset \DC(X)$ is an admissible subcategory, then there are two semi-orthogonal decompositions 
$$
\DC(X)=\langle \CA^{\perp},\CA\rangle=\langle \CA,{}^{\perp}\CA \rangle.
$$   
\end{ex}

Mutations are crucial operations to construct new full exceptional collections and semi-orthogonal decompositions.
Recall that the {\it left} and {\it right} mutation functors are defined by
$$
\begin{array}{cccl}
{\rm L}_{\CA}: & \DC(X) & \longrightarrow  & \DC(X)  \\
&F&\longmapsto &
{\rm Cone}(\imath\imath^{!}F\rightarrow F),
\end{array}
$$
and
$$
\begin{array}{cccl}
{\rm R}_{\CA}: & \DC(X) & \longrightarrow  & \DC(X)  \\
&F&\longmapsto &
{\rm Cone}(F\rightarrow \imath\imath^{\ast}F)[-1].
\end{array}
$$
In particular, if $\CA$ is generated by an exceptional object $E\in \DC(X)$, then   
$$
\xymatrix@C=0.5cm{
{\rm L}_{E}(F)={\rm Cone}(\RHom(E,F)\otimes E \rightarrow F)
}
$$
and
$$
\xymatrix@C=0.5cm{
{\rm R}_{E}(F)={\rm Cone}(F\rightarrow \RHom(F,E)^{\ast}\otimes E)[-1].
}
$$

\begin{defn}
We say that two full exceptional collections in $\DC(X)$
$$
\mathbb{E}=\{E_{1}, \cdots, E_{n}\}
\textrm{ and }\, 
\mathbb{F}=\{F_{1}, \cdots, F_{n}\}
$$ 
are {\it mutation-equivalent} if $\mathbb{F}$ can be obtained from $\mathbb{E}$ by a finite sequence of left or right mutations.    
\end{defn}

By a theorem of Bondal, every exceptional collection gives a semi-orthogonal decomposition:

\begin{thm}[{\cite[Theorem 3.2]{Bon90}}]
If $\{A_{1}, \cdots, A_{l}\}$ is an exceptional collection on $\DC(X)$,
then there exists a semi-orthogonal decomposition
$$
\DC(X)=\langle \CA_{X}, A_{1}, \cdots,A_{l} \rangle,
$$
where $\CA_{X}$ is the right orthogonal decomposition of $\langle A_{1}, \cdots, A_{l} \rangle $.    
\end{thm}

Given a semi-orthogonal decomposition $\DC(X)= \langle \CA_{1},\cdots, \CA_{l}\rangle$, 
there is a decomposition of the Grothendieck groups   
$$
K_{0}(\DC(X)) \cong K_{0}(\CA_{1}) \oplus \cdots \oplus K_{0}(\CA_{l}).
$$
In particular, if $\mathbb{A}:=\{A_{1},A_{2},\cdots,A_{l}\}$ is an exceptional collection on $\DC(Y)$, then 
$$
K_{0}(\DC(X)) \cong K_{0}(\CA)\oplus \ZN^{l},
$$
where $\CA$ is the right or left orthogonal complement of the sequence $\mathbb{A}$.
In particular, if $\DC(X)$ has a full exceptional collection of length $l$, then $K_{0}(\CA)=0$.

\begin{defn}
Let $\CA\subset \DC(X)$ be a non-trivial admissible subcategory.
We say that $\CA$ is called a \textit{phantom} if $K_{0}(\CA)=0$.
\end{defn}

Let $ \CA \subset \DC(X)$ and $ \CB \subset \DC(Y)$ be full triangulated subcategories. 
Then, the box tensor $\CA \boxtimes \CB \subseteq \DC(X\times Y) $ is defined to be the smallest full triangulated subcategory of $\DC(X\times Y)$ closed under direct summands and containing all objects of the form
$p_{X}^{\ast}A \otimes^{L} p_{Y}^{\ast}B$ for $ A \in \CA $ and $B \in \CB$. 
Here, $p_{X}:X\times Y\rightarrow X$ and $p_{Y}:X\times Y \rightarrow Y$ are natural projections.  

\begin{defn}
A phantom subcategory $\CA \subset  \DC(X) $ is called a {\it universal phantom} if for any smooth projective variety $Y$, the category $\CA \boxtimes \DC(Y) \subset \DC(X\times Y)$ is a phantom subcategory.
\end{defn}

We use $M(X)$ to denote the Chow motive of $X$ with integral coefficients, and 
$\mathbb{L}$ to denote the Lefschetz motive; for further details, we refer to \cite{Ma68} and \cite{GO13}. We say that a Chow motive $M(X)$ is of \textit{Lefschetz type} if it is isomorphic to a finite direct sum of motives of the form 
$\mathbb{L}^{\otimes r}$. 
The following result is then used to demonstrate that our phantom is universal.

\begin{prop}\label{prop:univ_phantom}
Let $\CA\subseteq \DC(X)$ be a phantom subcategory. 
If the Chow motive $M(X)$ has Lefschetz type, then $\CA$ is a universal phantom category.
In particular, if $X$ has a full exceptional collection of length $l$ and $\CA$ is orthogonal to a non-full exceptional collection of  length $l$, then $\CA$ is a universal phantom category.
\end{prop}

\begin{proof}
This is a combination of Corollary 4.3 and Proposition 4.4 in \cite{GO13}; see also \cite[Lemma 2.5]{KKL+26}.
\end{proof}

%-----------------------------------------------------------
\subsection{Height and pseudoheight of exceptional collections}

We will use the notion of height to detect the non-fullness of exceptional collections and normal Hochschild cohomology to distinguish between various phantom categories, both of which were introduced by Kuznetsov \cite{Kuz15}.

Let $X$ be a smooth complex projective variety and let $\CA\subset \DC(X)$ be an admissible subcategory.
Then, there is a semi-orthogonal decomposition
$\DC(X)=\langle \CA, \CB \rangle$. 
Recall that there is an induced semi-orthogonal decomposition 
$$
\DC(X\times X)=\Big\langle \CA\boxtimes \DC(X), \CB\boxtimes \DC(X)\Big\rangle.
$$ 
Thus, we have a distinguished triangle 
$$
\xymatrix@=0.5cm{
P_{\CB} \ar[r]& \Delta_{\ast}\CO_{X} \ar[r]& P_{\CA},}
$$
where $\Delta: X\hookrightarrow X\times X$ is the diagonal, $P_{\CA}\in \CA\boxtimes \DC(X)$ and $P_{\CB}\in \CB\boxtimes \DC(X)$.
The {\it Hochschild cohomology} of $X$ and $\CA$ are defined respectively by
\begin{align*}
\HH^{\bullet}(X)&:=\Ext^{\bullet}(\Delta_{\ast}\CO_{X},\Delta_{\ast}\CO_{X}) \\
\HH^{\bullet}(\CA)&:=\Ext^{\bullet}(P_{\CA},P_{\CA}).
\end{align*}
In particular, if $\CA=\DC(X)$, then $\HH^{\bullet}(X)=\HH^{\bullet}(\DC(X))$.

For one thing, the computation of Hochschild cohomology of an admissible subcategory is generally very difficult, so we need normal Hochschild cohomology. 
For another, Kuznetsov's notion of height relies on the normal Hochschild cohomology. 
Let us recall the definition.
Suppose that $\mathfrak{D}$ is the \u{C}ech enhancement of $\DC(X)$ and $\mathfrak{B} \subset \mathfrak{D}$ is a DG subcategory.

\begin{defn}[{\cite[Definition 3.2]{Kuz15}}]
The {\it normal Hochschild cohomology of $\mathfrak{B}$} in $\mathfrak{D}$ is defined as the derived tensor product of the diagonal and the normal bimodule of $\mathfrak{B}$:
$$
\NHH^{\bullet} (\mathfrak{B},\mathfrak{D}):=\mathfrak{B}\otimes^{\mathbb{L}}_{\mathfrak{B}^{\mathrm{op}}\otimes \mathfrak{B}} \mathfrak{D}_{\mathfrak{B}}^{-1}.
$$
Moreover, if $\mathfrak{B}$ is the induced enhancement of a triangulated subcategory $\mathcal{B} \subset \DC(X)$, we denote $\NHH^{\bullet}(\mathcal{B},X):=\NHH^{\bullet}(\mathfrak{B}, \mathfrak{D})$.
\end{defn}

The following theorem is useful for computing the Hochschild cohomology.

\begin{thm}[{\cite[Theorem 3.3]{Kuz15}}] \label{thm:nhh_triang}
Let $\DC(X)=\langle \CA, \CB \rangle$ be a semi-orthogonal decomposition.
Then there is a distinguished triangle
$$
\xymatrix@=0.5cm{
\NHH^{\bullet}(\CB,X) \ar[r]& \HH^{\bullet}(X) \ar[r]& \HH^{\bullet}(\CA).
}
$$
\end{thm}

From now on, let $\mathbb{E}:=\{E_{1}, \cdots, E_{n}\}$ be an exceptional collection on $X$.
We denote by $\mathcal{E} \subset \DC(X)$ the full triangulated subcategory generated by $E_{1},\cdots,E_{n}$.

\begin{defn}
The {\it height} $h(\mathbb{E})$ of the exceptional collection $\mathbb{E}$ is defined by
$$
h(\mathbb{E}):=\min \{ k\in \ZN \mid \mathrm{NHH}^{k}(\mathcal{E},X)\neq 0\}.
$$
\end{defn}

The significance of the height lies in the following criterion, which provides an effective tool for establishing the non-fullness of an exceptional collection.

\begin{lem}[{\cite[Proposition 6.1]{Kuz15}}]\label{not-full-criterion}
If the height $h(\mathbb{E})>0$, then the exceptional collection $\mathbb{E}$ is not full.
\end{lem}

To determine the height and the normal Hochschild cohomology, we use the following spectral sequence:

\begin{prop}[{\cite[Proposition 3.7]{Kuz15}}]\label{NHH--spectral-prop}
There exists a spectral sequence converging to the normal Hochschild cohomology
$$
E_{1}^{-p,q} \Rightarrow \NHH^{q-p}(\mathcal{E},X),
$$
where $E_{1}^{-p,q}$ is given by 
\begin{equation*}
\bigoplus_{\substack{1 \leq a_{0}<\cdots<a_{p} \leq n \\ k_{0}+\cdots+ k_{p}= q}} \Ext^{k_{0}}(E_{a_{0}}, E_{a_{1}}) \otimes \cdots \otimes \Ext^{k_{p-1}}(E_{a_{p-1}}, E_{a_{p}}) \otimes \Ext^{k_{p}}(E_{a_{p}}, \mathcal{S}^{-1}(E_{a_{0}})),
\end{equation*}
where $\mathcal{S}(-):=-\otimes \omega_{X}[\dim X]$ is the Serre functor of $\DC(X)$ and $\omega_{X}$ is the canonical sheaf of $X$.
Moreover, the differential on the $E_{1}$-page is given by the Yoneda product of Ext groups.
\end{prop}

In practice, we use 
the pseudoheight of an exceptional collection to determine the height.  
Recall that the {\it relative height} of two objects $F, G$ in $\DC(X)$ is defined by
$$
e(F,G):=\inf\{k\in \ZN \mid \ext^{k}(F,G)\neq 0\}.
$$
 
\begin{defn}\label{pseudoheight-defn} 
The {\it anticanonical pseudoheight} of the exceptional  collection $\mathbb{E}$ is defined by 
$$
\mathrm{ph}_{ac}(\mathbb{E}):=\min_{1\leq a_{0} < \cdots < a_{p} \leq n} \Big(\sum_{i=0}^{p-1}e(E_{a_{i}}, E_{a_{i+1}})+ e(E_{a_{p}}, E_{a_{0}} \otimes \omega_{X}^{-1})-p \Big),
$$
and the {\it pseudoheight} $\mathrm{ph}(\mathbb{E})$ of  $\mathbb{E}$ is defined by
$$
\mathrm{ph}(\mathbb{E}):= \mathrm{ph}_{ac}(\mathbb{E})+\dim X.
$$ 
\end{defn}

By Lemma \ref{not-full-criterion}, the following criterion is especially effective in practice for detecting the non-fullness of an exceptional collection.

\begin{lem}[{\cite[Lemma 4.5]{Kuz15}}]\label{ht-ge-ph-nonfull}
For any exceptional collection $\mathbb{E}$ on $X$, the height $h(\mathbb{E})\geq \mathrm{ph}(\mathbb{E})$.
In particular, if $\mathrm{ph}(\mathbb{E})>0$, then the exceptional collection $\mathbb{E}$ is not full. 
\end{lem}
 
%=====================================================================

\subsection{Linear systems and Dumnicki's diagram-cutting method}
Let $X$ be the blow-up of $\PB^{2}$ at $n\geq 10$ points $p_{1}, \cdots, p_{n}$ in general position.
Let $H$ be the pullback of the hyperplane class on $\PB^{2}$, and let $E_{i}$ be the exceptional divisor over $p_{i}$. 
Then, the Picard group 
$$
\Pic(X)=\ZN H\oplus \bigoplus_{i=1}^{n}\ZN E_{i}
$$
and the canonical divisor $K_{X}=-3H+\sum_{i=1}^{n} E_{i}$
with the intersection numbers
$$
H^{2}=1, E_{i}^{2}=-1, H\cdot E_{i}=0,\, 1\leq i\leq n, \,\textrm{ and }\, E_{i}\cdot E_{j}=0 \,(i\neq j).
$$
Therefore, for any divisor $D$ on $X$, it can be expressed as the form
$$
D=d H-\sum_{i=1}^{n} m_{i}E_{i},
$$
where $d, m_{i}\in \ZN$.
If $d>0$ and $m_{i}\geq 0$, $1\leq i\leq n$, then
$$
H^{0}(X,\CO_{X}(D))
\cong 
H^{0}(\PB^{2}, I_{p_{1}}^{m_{1}}I_{p_{2}}^{m_{2}}\cdots I_{p_{n}}^{m_{n}}(d)),
$$
where $I_{p_{i}}$ is the ideal sheaf of the point $p_{i}$ and $I_{p_{i}}^{m_{i}}$ is the $m_{i}$-th power of $I_{p_{i}}$. 
In this case, we denote by $\mathcal{L}(d,m_{1},\cdots,m_{n})$ the corresponding linear system of plane curves of degree $d$ through general points with the indicated multiplicities. 
If there are the same multiplicities, we denote
$\mathcal{L}(d,m_{i_{1}}^{k_{1}},\cdots,m_{i_{s}}^{k_{s}}):=\mathcal{L}(d,m_{1},\cdots,m_{n})$, where $k_{1}+\cdots+k_{s}=n$. 

In \cite{Dum07}, Dumnicki introduced the so-called {\it diagram-cutting method} for proving non-speciality of linear systems of plane curves.

\begin{defn}
Let $\Delta \subset \NN\times \NN$ be a finite subset (called a {\it diagram}), and $m_{1},\cdots, m_{n}\in \NN^{\ast}$. 
We define the linear system of curves $\mathcal{L}(\Delta,m_{1},\cdots,m_{n})$ to be the projective space of all
plane curves (that is, non-zero polynomials) generated by monomials with
exponents from $\Delta$ having multiplicities at least $m_{1},\cdots, m_{n}$ at $n$ points in general position.
\end{defn}

For an integer $d>0$, we define the diagram
$$
\Delta_{d}:=\{(a,b)\in \NN\times \NN \mid a+b\le d\}.
$$
Then, we have 
$$
\mathcal{L}(\Delta_{d},m_{1},\cdots,m_{n})=\mathcal{L}(d,m_{1},\cdots,m_{n})=|D|,
$$
where the divisor $D:=dH-\sum_{i=1}^{n} m_{i}E_{i}$.

\begin{defn}
Let $d, m_{1},\cdots, m_{n}\in \NN^{\ast}$.
A partition $\Delta_{d}=\mathrm{D}_{1}\cup\cdots\cup \mathrm{D}_{n}$ of the diagram $\Delta_{d}$ is called {\it affine of type $(d, m_{1}, \cdots, m_{n})$} if there exist affine functions 
$$
f_{i}:\RN^{2} \rightarrow \RN,\; 
f_{i}(x,y):=  
\alpha_{i}x+\beta_{i}y+\lambda_{i},\; \alpha_{i},\beta_{i},\lambda_{i}\in \RN,
$$
$R_{0}:=\Delta_{d}$ and $R_{i}:=R_{i-1}\cap\{(x,y)\mid f_{i}(x,y)<0\}$\, $(1\leq i\leq n-1)$ such that
$$
\mathrm{D}_{i}=R_{i-1}\cap\{(x,y)\mid f_{i}(x,y)>0\}
$$ 
and $\mathrm{D}_{n}=R_{n-1}$. 
We further require that $f_{i}$ does not vanish at any point of $R_{i-1}$ for $1 \leq i\leq n-1$.
\end{defn}
 
To identify divisors with no global sections, we mainly use the following Dumnicki's diagram-cutting method:

\begin{prop}\label{Dumnicki-method}
Let $\Delta_{d}=\mathrm{D}_{1}\cup\cdots\cup \mathrm{D}_{n}$ be an affine partition of type $(d,m_{1},\cdots,m_{n})$.
If, for every $1\leq i\leq n$, there exist $m_{i}$ horizontal (resp. vertical) lines  $\ell_{i,1},\cdots,\ell_{i,m_{i}}$
such that 
$$
\mathrm{D}_{i}\subset\bigcup_{k=1}^{m_{i}}\ell_{k}
\;
\textrm{ and }
\; 
\#(\mathrm{D}_{i}\cap \ell_{i,k})\leq k,\;  k=1,\cdots,m_{i},
$$
then $\mathcal{L}(\Delta_{d},m_{1},\cdots, m_{n})=\emptyset$.
\end{prop}

\begin{proof}
This is a direct consequence of \cite[Theorem 14]{Dum07}; see also \cite[\S 3.10, Fact 1, Fact 2 and Fact 3]{BBC+12}.
\end{proof}

Additionally, we shall frequently use the notion of Cremona transformations of linear systems:

\begin{defn} 
If the system $\mathcal{L}(d,m_{1},\cdots,m_{n})$ satisfies the conditions:
\begin{enumerate}
\item[(1)] $c:=m_{1}+m_{2}+m_{3}-d>0$,
\item[(2)] $m_{i} \geq c$, for $i=1,2,3$,
\end{enumerate}
then the system $\mathcal{L}(d-c,m_{1}-c, m_{2}-c, m_3-c,m_4, \cdots, m_{n})$ is called a {\it Cremona transformation} of $\mathcal{L}(d,m_{1},\cdots,m_{n})$ at points $p_{1}$, $p_{2}$ and $p_{3}$.
\end{defn}

In this case, we denote by $D:=dH-\sum_{i=1}^{n}m_{i} E_{i}$ and $D^{\prime}:=(d-c)H-\sum_{i=1}^{3}(m_{i}-c) E_{i}-\sum_{i=4}^{n}m_{i} E_{i}$ the corresponding divisors.
We also call $D^{\prime}$ a {\it Cremona transformation} of $D$.
Then, by Riemann--Roch theorem, we have $\chi(D)=\chi(D^{\prime})$. 
Moreover, the following result is well-known:

\begin{prop}\label{Cremnona-transform-consequ}
Suppose that $D^{\prime}$ is a {\it Cremona transformation} of $D$. 
Then 
$$
h^{0}(D)=h^{0}(D^{\prime}).  $$ 
\end{prop}

%===================================================================

\section{The desired full exceptional collections}\label{desired-FEC}

In this section, we review the construction of the desired full exceptional collections of line bundles on the blow-up of the complex projective plane $\PB^{2}$ at points in general position.
 
From now on, let $X={\rm Bl}_{p_{1},\cdots,p_{n}} \PB^{2}$
be the blow-up of the complex projective plane $\PB^{2}$ at $n$ points $p_{1},\cdots,p_{n}$ in general position, where $n\geq 10$.
Let $\ell \subset \PB^{2}$ be the line through $p_{1}$ and $p_{2}$, and let
$L$ be the strict transform of $\ell$. 
Then $L=H-E_{1}-E_{2}\in \Pic(X)$. 
Since the intersection number $L^{2}=(H-E_{1}-E_{2})^{2}=-1$, so $L$ is a $(-1)$-curve. 
Moreover, all the $(-1)$-curves $L, E_{3},\cdots, E_{n}$ are pairwise disjoint. Since ${\rm Bl}_{p_{1},p_{2}} \PB^{2} \cong {\rm Bl}_{p}\mathbf{F}_{0}$, by contracting the above $(-1)$-curves, we derive the blow-up of $\mathbf{F}_{0}$ at $n-1$ points 
$$
\rho:X\longrightarrow \mathbf{F}_{0},
$$
where $\mathbf{F}_{0}=\PB^{1}\times \PB^{1}$ is the $0$-th Hirzebruch surface.
Geometrically, this is the standard birational model obtained by blowing up $p_1,p_2$ and then contracting the strict transform of the line $\ell$. The additional exceptional curves $E_3,\cdots,E_{n}$ correspond to further blow-ups of points on $\mathbf{F}_{0}$.

We denote by $P$ and $Q$ the two ruling classes on the Hirzebruch surface $\mathbf{F}_{0}$. 
Then, we have
$$
\rho^{\ast}P=H-E_{1},
\;
\rho^{\ast}Q=H-E_{2},
$$
and hence $\rho^{\ast}(P+Q)=2H-E_{1}-E_{2}$.
Moreover, we have the standard full exceptional collection
$$
\DC(\mathbf{F}_{0})=\langle\CO,\CO(P),\CO(Q+aP),\CO(Q+(a+1)P)\rangle,
$$ 
where $a\geq 0$.
By pulling it back to $X$, it gives an admissible subcategory
\begin{eqnarray*}
L\rho^{\ast}\DC(\mathbf{F}_{0})
&=& 
\langle \CO_{X}, \CO_{X}(H-E_{1}), \CO_{X}((a+1)H-aE_{1}-E_{2}), \\
&& \;\;\;\;\; \CO_{X}((a+2)H-(a+1)E_{1}-E_{2})
\rangle    
\end{eqnarray*}
of the derived category $\DC(X)$.
Then one has: 

\begin{prop}
For every integer $a\geq 0$,
there is a full exceptional collection of line bundles
\begin{align}\label{Original-FEC}
\mathbb{D}_{a}&=\{
\CO_{X},\CO_{X}(H-E_{1}-E_{2}),\CO_{X}(E_{3}),\cdots,\CO_{X}(E_{n}), \CO_{X}(H-E_{1}), \nonumber \\ 
 & \;\;\;\;\;\; \CO_{X}((a+1)H-aE_{1}-E_{2}),\CO_{X}((a+2)H-(a+1)E_{1}-E_{2})
\}.     
\end{align}
\end{prop}

\begin{proof}
Since $\rho:X\rightarrow \mathbf{F}_{0}$ is the blow-up of $\mathbf{F}_{0}$ at the $n-1$ points corresponding to the exceptional divisors $L,E_{3},\cdots,E_{n}$, by Orlov's blow-up formula, one has a semi-orthogonal decomposition
$$
\DC(X)=
\langle
\CO_L(-1), \CO_{E_{3}}(-1),\cdots, \CO_{E_{n}}(-1), L\rho^{\ast}\DC(\mathbf{F}_{0})
\rangle.
$$
Therefore, $\DC(X)$ admits a semi-orthogonal decomposition
\begin{align*}
 \DC(X)= \; &
\langle \CO_L(-1),\CO_{E_{3}}(-1),\cdots,\CO_{E_{n}}(-1),
\CO_{X},
\CO_{X}(H-E_{1}),\\
 & \;\;\; \CO_{X}((a+1)H-aE_{1}-E_{2}),\CO_{X}((a+2)H-(a+1)E_{1}-E_{2})\rangle.   
\end{align*}

Now we will obtain the full exceptional collection by right mutating the objects past $\CO_{X}$. 
In fact, for any $(-1)$-curve $C\subset X$, there exists a short exact sequence
$$
\xymatrix@C=0.5cm{
0\ar[r]& \CO_{X} \ar[r] & \CO_{X}(C) \ar[r] & \CO_{C}(C) \ar[r] & 0.
}
$$
Since $C^{2}=-1$, one has $\CO_{C}(C)\cong \CO_{\PB^{1}}(-1)$. 
Thus, the above exact sequence becomes
$$
\xymatrix@C=0.5cm{
0\ar[r]& \CO_{X}
\ar[r]& \CO_{X}(C)
\ar[r]& \CO_{C}(-1)
\ar[r]& 0.
}
$$
Consequently, we have $\CO_{C}(-1)\in \langle \CO_{X},\CO_{X}(C)\rangle$ and $\CO_{X}(C)\in \langle \CO_{X},\CO_{C}(-1)\rangle$.
Therefore, two exceptional pairs $\{\CO_{C}(-1),\CO_{X}\}$ and $\{\CO_{X},\CO_{X}(C)\}$ generate the same admissible subcategory. 
Then, we will apply this to all the $(-1)$-curves $C=L,E_{3},\cdots,E_{n}$. 
Since $\CO_{X}(L)=\CO_{X}(H-E_{1}-E_{2})$, by mutating $\CO_{X}$ to the leftmost, the torsion sheaves
$$
\CO_{L}(-1),\CO_{E_{3}}(-1),\cdots,\CO_{E_{n}}(-1)
$$
are respectively replaced by the line bundles
$$
\CO_{X}(H-E_{1}-E_{2}),\CO_{X}(E_{3}),\cdots,\CO_{X}(E_{n}).
$$
As a result, the proposition follows.
\end{proof}

\begin{rem}\label{Original-FEC-mutation-equ}
For any two distinct integers $a\geq 0$, the full exceptional collections $\mathbb{D}_{a}$ are mutation-equivalent to each other. 
In fact, for every $a\geq 0$, the collections $\mathbb{D}_{a}$ and $\mathbb{D}_{a+1}$ differ by exactly one exceptional line bundle, so it is straightforward to see that they are related by a right mutation and hence are mutation-equivalent.
\end{rem}

Additionally, by Orlov's blow-up formula, there is a standard full exceptional collection of sheaves
\begin{equation}\label{standard-FEC}
 \langle \CO_{E_{1}}(-1), \cdots, \CO_{E_{n}}(-1), \CO_{X}, \CO_{X}(H), \CO_{X}(2H) \rangle.   
\end{equation}
Then, by taking the right mutations through $\CO_{X}$, it becomes the standard full exceptional collection of line bundles
\begin{equation}\label{standard-FEC-linebund}
\langle\CO_{X},\CO_{X}(E_{1}),\cdots,\CO_{X}(E_{n}),\CO_{X}(H),\CO_{X}(2H)\rangle.    
\end{equation}
For $n=10$ or $n=11$, respectively, this full exceptional collection of line bundles was considered prior to applying the involution \eqref{involution-Pic} in \cite{Kra24} and \cite{KKL+26}.
 
\begin{rem}\label{Ori-FEC-mut-equ-standard}
For each integer $a\geq 0$, the full exceptional collection \eqref{Original-FEC} is mutation-equivalent to \eqref{standard-FEC-linebund}.
In fact, by Remark \ref{Original-FEC-mutation-equ}, we take $a=0$ for example.
Since, for $j\geq 3$ and $i\neq j$, the exceptional pairs $\{\CO_{X}(H-E_{1}-E_{2}), \CO_{X}(E_{j})\}$ and $\{\CO_{X}(E_{i}) ,\CO_{X}(E_{j})\}$ are completely orthogonal, so \eqref{Original-FEC} becomes
$$
\left\langle \CO_{X}, \begin{smallmatrix} \CO_{X}(-E_{9}) \\
\cdots\\ \CO_{X}(-E_{n}) \end{smallmatrix}, \begin{smallmatrix}\CO_{X}(H-E_{1}-E_{2}) \\ \CO_{X}(E_{3}) \\ \cdots \\ \CO_{X}(E_{8})\end{smallmatrix}, \begin{smallmatrix} \CO_{X}(H-E_{1}) \\ \CO_{X}(H-E_{2})\end{smallmatrix}, \CO_{X}(2H-E_{1}-E_{2}) \right\rangle.
$$
Then, taking the left mutations through $\CO_{X}$, we have the following full exceptional collection
$$
\left\langle \begin{smallmatrix} \CO_{E_{9}}(-1) \\ \cdots \\ \CO_{E_{n}}(-1) \end{smallmatrix},  \overbrace{\CO_{X},  \begin{smallmatrix}\CO_{X}(H-E_{1}-E_{2}) \\ \CO_{X}(E_{3}) \\ \cdots \\ \CO_{X}(E_{8})\end{smallmatrix}, \begin{smallmatrix} \CO_{X}(H-E_{1}) \\ \CO_{X}(H-E_{2})\end{smallmatrix},\CO_{X}(2H-E_{1}-E_{2})}^{D^b(X_{1})} \right\rangle,
$$
where $X_{1}$ is the del Pezzo surface of degree $1$.
Since an arbitrary full exceptional collection of sheaves on a del Pezzo surface is mutation-equivalent to the standard full exceptional collection \eqref{standard-FEC} (see \cite[\S 6]{KO95}), so the above full exceptional collection is mutation-equivalent to 
$$
\left\langle \begin{smallmatrix} \CO_{E_{9}}(-1) \\ \cdots \\ \CO_{E_{n}}(-1) \end{smallmatrix},  \CO_{X},  \begin{smallmatrix} \CO_{X}(E_{1})   \\ \cdots \\ \CO_{X}(E_{8})\end{smallmatrix}, \CO_{X}(H), \CO_{X}(2H) \right\rangle.
$$
Finally, taking the right mutations through $\CO_{X}$, it follows that it is mutation-equivalent to \eqref{standard-FEC-linebund}.
\end{rem}

%=============================================================================

\section{Countably many phantoms on ten-point blow-up}\label{10pts-pf-mianthm1}

This section is devoted to the proof of Theorem \ref{mainthm}. 
Let $X$ be the blow-up of $\PB^{2}$ at $10$ points in general position.
Considering the involution
\begin{equation}\label{10pts-involution}
\begin{array}{cccl}
  \iota: & \mathrm{Pic}(X) & \longrightarrow  & \mathrm{Pic}(X)  \\
&D&\longmapsto &
-D-2(D\cdot K_{X})K_{X},
\end{array}
\end{equation}
the Riemann--Roch theorem gives 
$$
\chi(D)=\chi(\iota(D))
$$ 
for any divisor $D$ on $X$.
Applying the involution \eqref{10pts-involution} to the full exceptional collection \eqref{Original-FEC} for $n=10$,
for each integer $a\geq 0$, we have a numerically exceptional collection of line bundles of maximal length
\begin{equation}\label{NEC-10pts-case}
\{\CO_{X},\CO_{X}(A),\CO_{X}(B_{3}),\cdots,\CO_{X}(B_{10}),\CO_{X}(G),\CO_{X}(F_{a}),\CO_{X}(F_{a+1})\} 
\end{equation}
with the divisors 
$$
\begin{array}{r@{\hspace{3pt}}lr@{\hspace{3pt}}lcc}
 A:=&2K_{X}-H+E_{1}+E_{2}, & B_{j}:=&2K_{X}-E_{j},   \\
G:=&4K_{X}-H+E_{1},
& 
F_{a}:=&4(a+1)K_{X}-(a+1)H+aE_{1}+E_{2},
\end{array}
$$
where $3\leq j\leq 10$.
Moreover, the intersection numbers
$$
\begin{array}{r@{\hspace{3pt}}lr@{\hspace{3pt}}lc@{\hspace{3pt}}cc@{\hspace{3pt}}cc@{\hspace{3pt}}c}
A^{2} &=-1, &
A\ldotp K_{X}&=-1, &
A\ldotp B_{j}&=0, & 
A\ldotp G&=0, &
A\ldotp F_{a}&=0, \\
B_{j}^{2}&=-1, &
B_{j}\ldotp K_{X} &=-1, &
B_{j}\ldotp G &=0, &
B_{j}\ldotp F_{a} &=0, & \\
G^{2} & =0, &
G\ldotp K_{X} &=-2, &
G\ldotp F_{a} &=1, \\
F_{a}^{2} &=2a, &
F_{a}\ldotp K_{X}&=-2(a+1).
\end{array}
$$

\subsection{Divisors with no global sections}

To prove Theorem \ref{mainthm}, we first need to understand the following divisors, which have no global sections.

\begin{prop}\label{10pts-NoSect-case1}
For every integer $a\ge0$, 
let $D$ be one of the divisors
\begin{align*}
& (13a+4)H-(5a+3)E_{1}-\sum_{i=2}^{10}(4a+1)E_{i},  \\
& (13a+6)H-(5a+1)E_{1}-\sum_{i=2}^{10}(4a+2)E_{i}.
\end{align*}
Then $H^{0}(X,\CO_{X}(D))=0$.
\end{prop}

\begin{proof}
To prove Proposition \ref{10pts-NoSect-case1}, we use Proposition \ref{Dumnicki-method} to show that 
the corresponding linear systems 
$$
\mathscr{L}_{a}^{0}:=\mathcal{L}(13a+4,5a+3,(4a+1)^{9})
\;
\textrm{ and }
\;
\mathscr{L}_{a}^{1}:=\mathcal{L}(13a+6,5a+1,(4a+2)^{9})
$$
are empty, for every integer $a\ge0$.
We deal with the two systems simultaneously.  
Set $\sigma\in\{0,1\}$, and fix some notation:
$d=13a+4+2\sigma$,
$M=5a+3-2\sigma$,
$m=4a+1+\sigma$,
$q=5a+2$ 
and $T=15a+7-2\sigma$.
Then, we have $\mathscr{L}_{a}^{\sigma}=\mathcal{L}(d,M,m^{9})$.

{\bf Step 1: Construct an affine partition.} 
Following \cite[\S 3.10]{BBC+12},
for a small $0<\varepsilon<1$, we  define $9$ affine functions as follows:
\begin{align*}
f_{1}(x,y)&=-x-y+M-1+\varepsilon,&
f_{2}(x,y)&=x-(d-m)-1+\varepsilon,\\
f_{3}(x,y)&=y-(d-m)-1+\varepsilon,&
f_{4}(x,y)&=x-y-q-1+\varepsilon,\\
f_{5}(x,y)&=-x+y-q-1+\varepsilon,&
f_{6}(x,y)&=3x-y-3q-1+\varepsilon,\\
f_{7}(x,y)&=-3x+y+3a+\varepsilon,&
f_{8}(x,y)&=-x+y-a-1+\varepsilon,\\
f_{9}(x,y)&=x+3y-T+\varepsilon. 
\end{align*}
Since $0<\varepsilon<1$, so no lattice point lies on one of the separating affine lines.
We denote $R_{0}:=\Delta_{d}$ and $R_{i}=R_{i-1}\cap\{(x,y)\mid f_{i}(x,y)<0\}$.    
Then, the diagrams $\mathrm{D}_{1}, \cdots, \mathrm{D}_{10}$ are defined inductively by
\begin{equation}\label{partition-def1}
 \mathrm{D}_{i}:=R_{i-1}\cap \{(x,y)\mid f_{i}(x,y)>0\}, \; 1\le i\leq 9,
\end{equation}
and $\mathrm{D}_{10}:=R_{9}$.  
As a result, we obtain an affine partition:
$$
\Delta_{d}=\mathrm{D}_{1}\cup\cdots\cup \mathrm{D}_{10}.
$$
See for example, Figure \ref{affine-partition} for $a=2$, $\sigma=0$ and $\varepsilon=0.1$, the affine partition of $\Delta_{30}$:

\begin{figure}[H]
\centering
\begin{tikzpicture}[scale=0.18, font=\sffamily]
% x+y=30.9
\draw[line width=0.8pt] (0, 0) -- (30.9, 0) -- (0, 30.9) -- cycle;
% f3: y = 21.9 
\draw[semithick] (0, 21.9) -- (9.0, 21.9);
% f5: y = x + 12.9 
\draw[semithick] (0, 12.9) -- (9.0, 21.9);
% f2: x = 21.9 
\draw[semithick] (21.9, 0) -- (21.9, 9.0);
% f4: y = x - 12.9 
\draw[semithick] (12.9, 0) -- (21.9, 9.0);
% f1: y = -x + 12.1 
\draw[semithick] (0, 12.1) -- (12.1, 0);
% f7: y = 3x - 6.1 
\draw[semithick] (4.55, 7.55) -- (9.25, 21.65);
% f8: y = x + 2.9 
\draw[semithick] (4.5, 7.4) -- (14.0, 16.9);
% f6: y = 3x - 36.9 
\draw[semithick] (12.3, 0) -- (16.95, 13.95);
% f9: y = -x/3 + 12.3 
\draw[semithick] (7.52, 10.46) -- (14.76, 7.38);
% ============================================
\node at (3, 25.5) {\tiny $\mathrm{D}_{3}$};
\node at (4.5, 22.8) {\tiny $f_{3}$};
\node at (3, 18.3) {\tiny $\mathrm{D}_{5}$};
\node at (5.1, 19.4) {\tiny $f_{5}$};
\node at (4, 13) { \tiny $\mathrm{D}_{7}$};
\node at (6.1, 15.7) {\tiny$f_{7}$};
\node at (2.6, 4) { \tiny $\mathrm{D}_{1}$};
\node at (6.5, 4) {\tiny $f_{1}$};
\node at (9.9, 17) {\tiny $\mathrm{D}_{8}$};
\node at (8.5, 13) {\tiny $f_{8}$};
\node at (13, 12) {\tiny $\mathrm{D}_{9}$};
\node at (11.3, 9.8) {\tiny $f_{9}$};
\node at (9.7, 5.5) {\tiny $\mathrm{D}_{10}$};
\node at (17.7, 7.8) {\tiny $\mathrm{D}_{6}$};
\node at (17, 10) {\tiny $f_{6}$};
\node at (19.5, 3) {\tiny $\mathrm{D}_{4}$};
\node at (17, 3) {\tiny $f_{4}$};
\node at (25, 3) {\tiny $\mathrm{D}_{2}$};
\node at (23, 5) {\tiny $f_{2}$};
\end{tikzpicture}
\caption{An affine partition of $\Delta_{30}$}
\label{affine-partition}
\end{figure}
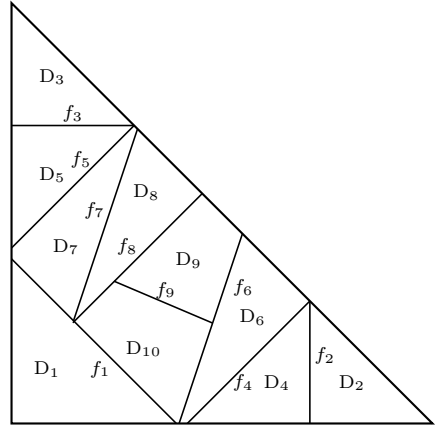

{\bf Step 2: 
Verify the conditions of Proposition \ref{Dumnicki-method}.}
For all the diagrams $\mathrm{D}_{1},\cdots, \mathrm{D}_{10}$, we will discuss case-by-case as follows.

First of all, for the diagrams $\mathrm{D}_{1}$, $\mathrm{D}_{2}$, $\mathrm{D}_{3}$, $\mathrm{D}_{4}$ and $\mathrm{D}_{5}$, it follows directly from \eqref{partition-def1} that we have:
\begin{enumerate}
\item[(1)] For the diagram $\mathrm{D}_{1}$, we have
$
\displaystyle 
\mathrm{D}_{1} =\bigcup_{r=1}^{M} \{(r-1,y)\mid 0\leq y\leq M-r\}.
$
Thus, the cardinality $\# \mathrm{D}_{1}=\frac{M(M+1)}{2}$. 
We take $M$ vertical lines $\ell_{1,k}:=\{x=M-k\}$, 
$1\leq k\leq M$.
Then, the cardinality $\#(\mathrm{D}_{1}\cap \ell_{1,k})=k$, $1\leq k\leq M$.

\item[(2)] For the diagram $\mathrm{D}_{2}$,
we have
$\displaystyle 
\mathrm{D}_{2} =\bigcup_{r=1}^{m}
 \{(d-m+r,y) \mid 0\leq y\leq m-r\}.
$
Thus, the cardinality $\# \mathrm{D}_{2}=\frac{m(m+1)}{2}$.
We take $m$ vertical lines $\ell_{2,k}:=\{x=d-k+1\}$, 
$1\leq k\leq m$.
Then the cardinality $\#(\mathrm{D}_{2}\cap \ell_{2,k})=k$, $1\leq k\leq m$.

\item[(3)] For the diagram $\mathrm{D}_{3}$, 
we have
$\displaystyle 
\mathrm{D}_{3} =\bigcup_{r=1}^{m}
 \{(x,d-m+r)\mid 0\leq x\leq m-r\}.
$
Thus, the cardinality $\# \mathrm{D}_{3}=\frac{m(m+1)}{2}$.
We take $m$ horizontal lines $\ell_{3,k}:=\{y=d-k+1\}$, 
$1\leq k\leq m$.
Then, the cardinality $\#(\mathrm{D}_{3}\cap \ell_{3,k})=k$, $1\leq k\leq m$.

\item[(4)] For the diagram $\mathrm{D}_{4}$,
we have 
$\displaystyle 
\mathrm{D}_{4}=\bigcup_{r=1}^{m}
 \{(q+r,y)\mid 0\leq y\leq r-1\}.
$
Thus, the cardinality $\# \mathrm{D}_{4}=\frac{m(m+1)}{2}$.
We take $m$ vertical lines $\ell_{4,k}:=\{x=q+k\}$, 
$1\leq k\leq m$.
Then, the cardinality $\#(\mathrm{D}_{4}\cap \ell_{4,k})=k$, $1\leq k\leq m$.

\item[(5)] For the diagram $\mathrm{D}_{5}$, 
we have 
$\displaystyle 
\mathrm{D}_{5} =\bigcup_{r=1}^{m}
 \{(x,q+r) \mid 0\leq x\leq r-1\}.
$
Thus, the cardinality $\# \mathrm{D}_{5}=\frac{m(m+1)}{2}$.
We take $m$ horizontal lines $\ell_{5,k}:=\{y=q+k\}$, 
$1\leq k\leq m$.
Then, the cardinality $\#(\mathrm{D}_{5}\cap \ell_{5,k})=k$, $1\leq k\leq m$.
\end{enumerate}

Moreover, for the diagrams $\mathrm{D}_{6}$, $\mathrm{D}_{7}$ and $\mathrm{D}_{8}$, the discussions are the same:
\begin{enumerate}
\item[(i)]For the diagram $\mathrm{D}_{6}$, the conditions of not belonging to $\mathrm{D}_{4}$, belonging to $\mathrm{D}_{6}$, and lying in $\Delta_{d}$ yield the condition:
$
x-q\leq y\leq \min\{3x-3q-1,d-x\}.
$
Thus, we have 
$$ \displaystyle 
\mathrm{D}_{6}=\bigcup_{r=1}^{m}
 \{(q+r,y) \mid r\leq y\leq\min\{3r-1,2m-r\}\}.
$$
We take $m$ vertical lines as follows: (i) if $k$ is even, then $\ell_{6,k}:=\{x=q+\frac{k}{2}\}$; (ii) if $k$ is odd, then $\ell_{6,k}:=\{x=q+m-\frac{k-1}{2}\}$.
Thus, the cardinality $\#(\mathrm{D}_{6}\cap \ell_{6,k})=k$, $1\leq k\leq m$, and then $\# \mathrm{D}_{6}=\frac{m(m+1)}{2}$.

\item[(ii)] For the diagram $\mathrm{D}_{7}$, after simplifying the preceding inequalities one obtains the constraint conditions:
$ 1-\sigma\leq x\leq m-\sigma$ and $\max\{M-x,3x-3a\}\leq y\leq q+x$.
Thus, we get
$$\displaystyle 
\mathrm{D}_{7}=\bigcup_{r=1}^{m}
  \{(r-\sigma,y) \mid \max\{M-r+\sigma,3r-3\sigma-3a\}\le y \leq q+r-\sigma
 \}.
$$
We take $m$ vertical lines as follows: (i) if $k$ is even, then $\ell_{7,k}=\{x=\frac{k}{2}-\sigma\}$; (ii) if $k$ is odd, then $\ell_{7,k}=\{x=m-\frac{k-1}{2}-\sigma\}$.
Thus, the cardinality $\#(\mathrm{D}_{7}\cap \ell_{7,k})=k$, $1\leq k\leq m$, and then $\# \mathrm{D}_{7}=\frac{m(m+1)}{2}$.

\item[(iii)] For the diagram $\mathrm{D}_8$, the constraint conditions are
$2a+1\leq x \leq 2a+m$ and $x+a+1\leq y \leq \min\{3x-3a-1,d-x\}$.
Hence, we have
$$\displaystyle 
\mathrm{D}_{8}=\bigcup_{r=1}^{m}
 \left\{(2a+r,y) \mid
 3a+r+1\leq y\leq
 \min\{3a+3r-1,11a+4+2\sigma-r\}
 \right\}.
$$
We take $m$ vertical lines as follows: (i) if $k$ is even, then $\ell_{8,k}=\{x=2a+m+1-\frac{k}{2}\}$; (ii) if $k$ is odd, then $\ell_{8,k}=\{x=2a+\frac{k+1}{2}\}$.
Thus, the cardinality $\#(\mathrm{D}_{8}\cap \ell_{8,k})=k$, $1\leq k\leq m$, and $\# \mathrm{D}_{8}=\frac{m(m+1)}{2}$.
\end{enumerate}

Finally, for the last two diagrams, the discussions are also the same:
\begin{enumerate}
\item[(1)]For the diagram $\mathrm{D}_{9}$,
set $x=3a+1+r$. The new inequality $x+3y\geq T$, together
with the complements of the $\mathrm{D}_{6}$- and $\mathrm{D}_{8}$-inequalities and the boundary
of $\Delta_{d}$, gives the condition: 
$l_{9}(r) \leq y \leq u_{9}(r)$,
where 
$l_{9}(r):=\max \{
 4a+2-\left\lfloor\frac{r+2\sigma}{3}\right\rfloor,
 3r-6a-3 \}$ and 
 $u_{9}(r):=\min\{4a+1+r,\,10a+3+2\sigma-r\}$.
Thus, we have
$$\displaystyle 
\mathrm{D}_{9}=\bigcup_{r=1}^{m}
 \{(3a+1+r,y)\mid l_{9}(r)\leq y\leq u_{9}(r)\}.
$$
We take $m$ vertical lines as follows:
\begin{enumerate}
\item[(i)] If $\sigma=0$, then there are four cases:
(1) if $k\equiv 0\; ({\rm mod}\; 4)$, then $\ell_{9,k}=\{x=3a+1+\frac{3}{4}k\}$;
(2) if $k\equiv 1\; ({\rm mod}\; 4)$, then $\ell_{9,k}=\{x=3a+1+\frac{3k+1}{4}\}$;
(3) if $k\equiv 2\;({\rm mod}\; 4)$, then $\ell_{9,k}=\{x=3a+1+\frac{3k+2}{4}\}$;
(4) if $k\equiv 3\; ({\rm mod}\; 4)$, then $\ell_{9,k}=\{x=3a+1+m-\frac{k-3}{4}\}$.
Thus, the cardinality$\#(\mathrm{D}_{9}\cap\ell_{9,k})=k$, $1\leq k\leq m$ and $\#\mathrm{D}_{9}=\frac{m(m+1)}{2}$.
\item[(ii)] If $\sigma=1$, then there exist four cases:
(1) if $k\equiv 0\; ({\rm mod}\; 4)$, then $\ell_{9,k}=\{x=3a+1+\frac{3}{4}k\}$;
(2) if $k\equiv 1\; ({\rm mod}\; 4)$, then $\ell_{9,k}=\{x=3a+1+m-\frac{k-1}{4}\}$;
(3) if $k\equiv 2\; ({\rm mod}\; 4)$, then $\ell_{9,k}=\{x=3a+1+\frac{3k-2}{4}\}$;
(4) if $k\equiv 3\; ({\rm mod}\; 4)$, then $\ell_{9,k}=\{x=3a+1+\frac{3k-1}{4}\}$.
Thus, the cardinality$\#(\mathrm{D}_{9}\cap\ell_{9,k})=k$, $1\leq k\leq m$ and $\#\mathrm{D}_{9}=\frac{m(m+1)}{2}$.
\end{enumerate}

\item[(2)]For the diagram $\mathrm{D}_{10}$, on the horizontal line
$y=r-\sigma$, the inequalities defining the residual piece $\mathrm{D}_{10}$
reduce to $ l_{10}(r) \leq x \leq u_{10}(r)$,
where 
$l_{10}(r):=\max\{q+1-\sigma-r,\,r-\sigma-a\}$ 
and
$u_{10}(r):=\min\left\{
 q+\left\lfloor\frac{r-\sigma}{3}\right\rfloor,
 15a+6+\sigma-3r \right\}$.
Thus, we have
$$ \displaystyle 
\mathrm{D}_{10}
=\bigcup_{r=1}^{m}
 \{(x,r-\sigma)\mid l_{10}(r)\leq x\leq u_{10}(r)\}.
$$
We take $m$ horizontal lines as follows:
\begin{enumerate}
\item[(i)] If $\sigma=0$, then there are four cases:
(1) if $k\equiv 0\; ({\rm mod}\; 4)$, then $\ell_{10,k}=\{y=\frac{3}{4}k\}$;
(2) if $k\equiv 1\; ({\rm mod}\; 4)$, then $\ell_{10,k}=\{y=\frac{3k+1}{4}\}$;
(3) if $k\equiv 2\;({\rm mod}\; 4)$, then $\ell_{10,k}=\{y=\frac{3k+2}{4}\}$;
(4) if $k\equiv 3\; ({\rm mod}\; 4)$, then $\ell_{10,k}=\{y=m-\frac{k-3}{4}\}$.
Thus, the cardinality$\#(\mathrm{D}_{10}\cap\ell_{10,k})=k$, $1\leq k\leq m$ and $\#\mathrm{D}_{10}=\frac{m(m+1)}{2}$.
\item[(ii)] If $\sigma=1$, then there exist four cases:
(1) if $k\equiv 0\; ({\rm mod}\; 4)$, then $\ell_{10,k}=\{y=\frac{3}{4}k-1\}$;
(2) if $k\equiv 1\; ({\rm mod}\; 4)$, then $\ell_{10,k}=\{y=m-\frac{k-1}{4}-1\}$;
(3) if $k\equiv 2\; ({\rm mod}\; 4)$, then $\ell_{10,k}=\{y=\frac{3k-2}{4}-1\}$;
(4) if $k\equiv 3\; ({\rm mod}\; 4)$, then $\ell_{10,k}=\{y=\frac{3k-1}{4}-1\}$.
Thus, the cardinality$\#(\mathrm{D}_{10}\cap\ell_{10,k})=k$, $1\leq k\leq m$ and $\#\mathrm{D}_{10}=\frac{m(m+1)}{2}$.
\end{enumerate}
\end{enumerate}
Consequently, based on the above discussions, Proposition \ref{Dumnicki-method} immediately yields $H^{0}(X,\CO_{X}(D))=0$.  
\end{proof}

\begin{cor}\label{10pts-NoSect-case1-cor}
Let $D$ be one of the divisors
\begin{align*} 
& (13a+7)H-(5a+2)E_{1}-\sum_{i=2}^{3}(4a+3)E_{i}-\sum_{j=4}^{10} (4a+2) E_{j},\; a\geq 0,\\
& 13aH-(5a-1)E_{1}-(4a+1)E_{2}-\sum_{i=3}^{10} 4a E_{i},\; a\geq 1.
\end{align*}
Then $H^{0}(X,\CO_{X}(D))=0$.
\end{cor}

\begin{proof}
For one thing, for each integer $a\geq 0$, we denote by 
$$
\mathscr{J}_{a}:=\mathcal{L}(13a+7,5a+2,(4a+3)^2,(4a+2)^{7})
$$ 
the corresponding linear system.
Then, applying the Cremona transformation of $\mathscr{J}_{a}$ at points $p_{1}$, $p_{2}$ and $p_{3}$, we obtain its Cremona transformation
$\mathscr{L}_{a}^{1}$. 

For another, we use $\mathscr{S}_{a}$ to denote the linear system $\mathcal{L}(13a,5a-1,4a+1,(4a)^{8})$, $a\geq 1$.
We begin with the linear system 
$$
\mathscr{L}_{a}^{0}=\mathcal{L}(13a+4,5a+3,(4a+1)^{9}).
$$
Then, applying the Cremona transformation of $\mathscr{L}_{a}^{0}$ at points $p_{1}$, $p_{2}$ and $p_{3}$, we obtain its Cremona transformation
$$
\mathcal{L}(13a+3,5a+2,(4a)^{2},(4a+1)^{7}).
$$
Next, applying three times Cremona transformations of the above linear system, at every stage, use the point $p_{1}$ and two as-yet unused points of multiplicity $4a+1$, up to permutations, we get the linear system $\mathscr{S}_{a}$.

In summary, according to Proposition \ref{Cremnona-transform-consequ} and Proposition \ref{10pts-NoSect-case1},  we obtain $H^{0}(X,\CO_{X}(D))=0$.
\end{proof}

%----------------------------------------------------------------------
\subsection{Proof of Theorem \ref{mainthm}}

The proof consists of the following two propositions.

\begin{prop}\label{EC-10pts-main-prop}
For every integer $a\geq 0$, the sequence
\begin{equation}\label{EC-sequ-10pts-main}
\{\CO_{X},\CO_{X}(A),\CO_{X}(B_{3}),\cdots,\CO_{X}(B_{10}),\CO_{X}(G),\CO_{X}(F_{a}),\CO_{X}(F_{a+1})\} 
\end{equation}
is an exceptional collection of line bundles.
\end{prop}

\begin{proof}
Since the sequence \eqref{EC-sequ-10pts-main} is numerically exceptional, by the definition of Euler characteristic,
it is sufficient to show the following dimensions of $\Hom$-spaces and $\Ext^{2}$-spaces are zero: 
\begin{align*}
\hom(\CO_{X}(A),\CO_{X}) =\; &  h^{0}(-A),   \\
\ext^{2}(\CO_{X}(A),\CO_{X}) =\; & h^{2}(-A)= h^{0}(K_{X}+A), \\
\hom(\CO_{X}(B_{j}),\CO_{X}) =\; & h^{0}(-B_{j}), \\ 
\ext^{2}(\CO_{X}(B_{j}),\CO_{X}) =\; & h^{2}(-B_{j})= h^{0}(K_{X}+B_{j}) \\
\hom(\CO_{X}(G),\CO_{X}) =\; &h^{0}(-G), \\
\ext^{2}(\CO_{X}(G),\CO_{X}) =\; & h^{2}(-G)=h^{0}(K_{X}+G), \\
\hom(\CO_{X}(B_{j}),\CO_{X}(A)) =\; & h^{0}(A-B_{j}), \\
\ext^{2}(\CO_{X}(B_{j}),\CO_{X}(A)) =\; & h^{2}(A-B_{j})=h^{0}(K_{X}-A+B_{j}), \\
\hom(\CO_{X}(G),\CO_{X}(A)) =\; & h^{0}(A-G), \\
\ext^{2}(\CO_{X}(G),\CO_{X}(A)) =\; & h^{2}(A-G)=h^{0}(K_{X}-A+G), \\
\hom(\CO_{X}(B_{j}),\CO_{X}(B_{i})) =\; & h^{0}(B_{i}-B_{j}), \\
\ext^{2}(\CO_{X}(B_{j}),\CO_{X}(B_{i})) =\; & h^{2}(B_{i}-B_{j})=h^{0}(K_{X}-B_{i}+B_{j}), \\
\hom(\CO_{X}(G),\CO_{X}(B_{j})) =\; & h^{0}(B_{j}-G), \\
\ext^{2}(\CO_{X}(G),\CO_{X}(B_{j}))
=\; & h^{2}(B_{j}-G)=h^{0}(K_{X}-B_{j}+G), \\
\hom(\CO_{X}(F_{a}),\CO_{X})   =\; & h^{0}(-F_{a}), \\
\ext^{2}(\CO_{X}(F_{a}),\CO_{X}) =\; & h^{2}(-F_{a})=h^{0}(K_{X}+F_{a}), \\
\hom(\CO_{X}(F_{a}),\CO_{X}(A)) =\; & h^{0}(A-F_{a}), \\
\ext^{2}(\CO_{X}(F_{a}),\CO_{X}(A)) =\; & h^{2}(A-F_{a})=h^{0}(K_{X}-A+F_{a}), \\
\hom(\CO_{X}(F_{a}),\CO_{X}(B_{j})) =\; & h^{0}(B_{j}-F_{a}), \\
\ext^{2}(\CO_{X}(F_{a}),\CO_{X}(B_{j}))=\; & h^{2}(B_{j}-F_{a})=h^{0}(K_{X}-B_{j}+F_{a}), \\
\hom(\CO_{X}(F_{a}),\CO_{X}(G))  =\; & h^{0}(G-F_{a}), \\
\ext^{2}(\CO_{X}(F_{a}),\CO_{X}(G)) =\; & h^{2}(G-F_{a})=h^{0}(K_{X}-G+F_{a}), \\
\hom(\CO_{X}(F_{a+1}),\CO_{X}(F_{a}))=\; &h^{0}(F_{a}-F_{a+1}), \\
\ext^{2}(\CO_{X}(F_{a+1}),\CO_{X}(F_{a}))=\;& h^{2}(F_{a}-F_{a+1})=h^{0}(K_{X}-F_{a}+F_{a+1}),
\end{align*}
where $3\leq i, j\leq 10$ and $i<j$.

For the $\Ext^{2}$-spaces, 
since the divisors $K_{X}+A$, $K_{X}+B_{j}$, $K_{X}+G$, $K_{X}-A+B_{j}$, $K_{X}-A+G$, $K_{X}-B_{i}+B_{j}$, $K_{X}-B_{j}+G$, $K_{X}+F_{a}$, $K_{X}-A+F_{a}$, $K_{X}-B_{j}+F_{a}$, $K_{X}-G+F_{a}$ and $K_{X}-F_{a}+F_{a+1}$ intersect negatively with the nef divisor $H$, so 
all the above $\Ext^{2}$-spaces are trivial.

For the $\Hom$-spaces, 
since $(A-B_{j}).H=-1$ and $B_{i}-B_{j}=E_{j}-E_{i}$, so $\hom(B_{j},A)=0$ for $3\leq j\leq 10$ and $\hom(B_{j},B_{i})=0$ for $3\leq i<j\leq 10$. 
The remaining cases correspond to the divisors:
\begin{align*}
-A&=7H-3E_{1}-3E_{2}-\sum_{i=3}^{10}2E_{i}, \qquad \quad  \quad \quad
-B_{j} = 6H-E_{j}-\sum_{i\neq j}2 E_{i},\\
-G&=13H-5E_{1}-\sum_{i=2}^{10} 4E_{i},
\qquad \quad \quad \quad \quad \;\;
A-G =6H-E_{2}-\sum_{i\neq 2}2 E_{i}, \\
B_{j}-G&= 7H-3E_{1}-3E_{j}-\sum_{i\neq 1, j}^{10}2E_{i}, 
\quad \quad \,
F_{a}-F_{a+1} =-G,\\
-F_{a}&=(13a+13)H-(5a+4)E_{1}-(4a+5)E_{2}-\sum_{i=3}^{10} (4a+4)E_{i}, \\
A-F_{a} &= (13a+6)H-(5a+1)E_{1}-\sum_{i=2}^{10} (4a+2)E_{i},\\
B_{j}-F_{a}&=(13a+7)H-(5a+2)E_{1}-(4a+3)(E_{2}+E_{j})-\sum_{i\neq 1,2,j} (4a+2)E_{i}, \\
G-F_{a}&=13aH-(5a-1)E_{1}-(4a+1)E_{2}-\sum_{i=3}^{10} 4aE_{i},
\end{align*}
where $a\geq 0$ and $3\leq j\leq 10$.
By Proposition \ref{10pts-NoSect-case1}, we have $h^{0}(A-F_{a})=0$.
Up to permutations, Proposition \ref{10pts-NoSect-case1} yields that $h^{0}(-B_{j})=0$ and $h^{0}(A-G)=0$.
Moreover, by Corollary \ref{10pts-NoSect-case1-cor}, we get $h^{0}(-F_{a})=0$ and
$h^{0}(G-F_{a})=0$.
Finally, up to permutations, by Corollary \ref{10pts-NoSect-case1-cor}, we have $h^{0}(-A)=0$, $h^{0}(B_{j}-G)=0$, $h^{0}(-G)=h^{0}(F_{a+1}-F_{a})=0$ and $h^{0}(B_{j}-F_{a})=0$.
Consequently, the sequence \eqref{EC-sequ-10pts-main} is an exceptional collection of line bundles of maximal length.
\end{proof}

To finish the proof of Theorem \ref{mainthm}, by Proposition \ref{prop:univ_phantom}, it is sufficient to prove that the exceptional collection \eqref{EC-sequ-10pts-main} is not full.
For obtaining the non-fullness, we will show that the height of the exceptional collection \eqref{EC-sequ-10pts-main} is positive. 
Given an integer $a\geq 0$, for convenience, we set 
$$
L_{1}:=\CO_{X}, \; L_{2}:=\CO_{X}(A), \; L_{l}:=\CO_{X}(B_{j}),\;
L_{11}:=\CO_{X}(G),
$$
$$
L_{12}:=\CO_{X}(F_{a}),
L_{13}:=\CO_{X}(F_{a+1}),
$$
where $3\leq j \leq 10$.
Then, all the forward $\Hom$-spaces are trivial:

\begin{lem}\label{10pts-forward-hom-zero-main}
$\Hom(L_{i},L_{j})=0$ for $1\leq i<j\leq 13$. 
\end{lem}

\begin{proof} 
Since $H$ is a nef divisor,
by a direct computation, 
every line bundle $L_{j} \otimes L_{i}^{-1}$ has no global sections, since it has negative intersection numbers with $\CO_{X}(H)$ except for 
$$
\CO_{X}(B_{j}-A), \CO_{X}(B_{j}-B_{i})
\;
\textrm{ and } 
\;
\CO_{X}(F_{0}-G),
$$
where $3\leq j\leq 10$ and $3\leq i<j\leq 10$.
Since $B_{j}-A=H-E_{1}-E_{2}-E_{j}$ and three general points $p_{1}, p_{2}, p_{j}$ are not collinear, so $h^{0}(B_{j}-A)=0$. Moreover, note that $\hom(\CO_{X}(B_{i}),\CO_{X}(B_{j}))= h^{0}(E_{i}-E_{j})=0$ and $\hom(\CO_{X}(G),\CO_{X}(F_{0}))=h^{0}(E_{2}-E_{1})=0$.
This concludes the proof.
\end{proof}

Now we are ready to prove that the exceptional collection \eqref{EC-sequ-10pts-main} is not full.

\begin{prop}\label{ph-10pts-onemore}
For every integer $a\geq 0$,
the height of the exceptional collection \eqref{EC-sequ-10pts-main} is positive. In particular, the exceptional collection \eqref{EC-sequ-10pts-main} is not full.
\end{prop}

\begin{proof}
For convenience, we set the exceptional collection
$\mathbb{A}_{a}:=\{L_{1}, \cdots, L_{13}\}$.
Then, the anticanonical pseudoheight of $\mathbb{A}_{a}$ is
$$
\phac(\mathbb{A}_{a})=\min_{1\leq a_{0}<\cdots<a_{p}\leq 13}\Big(\sum_{i=0}^{p-1} e(L_{a_{i}},L_{a_{i+1}})+e(L_{a_{p}},L_{a_{0}}\otimes \omega_{X}^{-1})-p\Big).
$$
By Lemma \ref{10pts-forward-hom-zero-main}, for any $1\leq i<j\leq 13$, the relative height $e(L_{i},L_{j})\geq 1$.
Note that 
$e(L_{a_{p}},L_{a_{0}}\otimes \omega_{X}^{-1})\geq 0$. 
Thus, the anticanonical pseudoheight
$\phac(\mathbb{A}_{a})\geq 0$.
Therefore, we have
$$
h(\mathbb{A}_{a})\geq \ph(\mathbb{A}_{a})=\phac(\mathbb{A}_{a})+\dim X\geq 2.
$$
Then, by Lemma \ref{ht-ge-ph-nonfull}, the exceptional collection \eqref{EC-sequ-10pts-main} is not full.
\end{proof}

\begin{rem}
Let $X$ be the blow-up of $\PB^{2}$ at $10$ points in general position.
Then, by Orlov's blow-up formula, Theorem \ref{mainthm} implies that any blow-up of $X$ at finite points has countably infinitely many phantoms. 
In particular, the blow-up of $\PB^{2}$ at $n$ points in general position has countably infinitely many phantoms, $n\geq 11$.
\end{rem}

%=====================================================================

\section{Comparison of phantoms}\label{Compar-phantoms}

In this section, we prove Theorem \ref{distinct-phantoms-thm}.
More precisely, we distinguish the various phantom categories on smooth rational surfaces by determining the Hochschild cohomology of the phantom categories in Theorem \ref{mainthm}.

Let $X$ be the blow-up of $\PB^{2}$ at $10$ points in general position. 
For every integer $a\geq 0$, we adopt the following notation for convenience:
$$
L_{1}:=\CO_{X}, \; 
L_{2}:=\CO_{X}(A), \; 
L_{j}:=\CO_{X}(B_{j}) \;\; (3\leq j \leq 10), 
$$
and
$$
L_{11}:=\CO_{X}(G), \; 
L_{12}:=\CO_{X}(F_{a}), \; 
L_{13}:=\CO_{X}(F_{a+1}).
$$
We begin by presenting two propositions concerning the Ext groups that will be used in the subsequent arguments.

\begin{prop}\label{10pts-ext2-main}
For any $1\leq i<j\leq 13$,
we have $\ext^{k}(L_{i},L_{j})=0$ for $k\neq 2$ and 
$$
\ext^{2}(L_{i},L_{j})=\chi(L_{i},L_{j})=\chi(L_{j}\otimes L_{i}^{-1}).
$$
\end{prop}

\begin{prop}\label{10pts-ext1-main}
For any $1 \leq i< j\leq 13$, we have $\ext^{k}(L_{j},L_{i}\otimes \omega_{X}^{-1})=0$ for $k\neq 1$ and 
$$
\ext^{1}(L_{j},L_{i}\otimes \omega_{X}^{-1})=-\chi(L_{j},L_{i}\otimes \omega_{X}^{-1})=-\chi(L_{i}\otimes L_{j}^{-1}\otimes \omega_{X}^{-1}).
$$
\end{prop}

\subsection{Proof of Proposition \ref{10pts-ext2-main}}

We first have an easy observation:

\begin{lem}\label{cubic-criterion}
Let $D$ be a divisor on $X$ such that $h^{2}(D)=0$.
If there is a smooth elliptic curve $C\subset X$ such that $h^{0}(D-C)=0$ and $D\ldotp C=\chi(D)>0$.
Then $h^{0}(D)=\chi(D)$ and $h^{1}(D)=0$.
\end{lem}

\begin{proof}
Twisting the structure sheaf exact sequence of $C\subset X$ by the line bundle $\CO_{X}(D)$, 
we obtain the following exact sequence
\begin{equation}\label{twisting-SEC-elliptic}
\xymatrix@C=0.5cm
{
0\ar[r]& \CO_{X}(D-C)\ar[r] & \CO_{X}(D)\ar[r] & \CO_{C}(D) \ar[r] & 0.
}    
\end{equation}
Since $C$ is a smooth elliptic curve and $D\ldotp C=\chi(D)>0$, by Riemann--Roch theorem on $C$, we have 
$h^{0}(C,\CO_{C}(D))=D\ldotp C=\chi(D)$.
Since $h^{0}(D-C)=0$, by the exact sequence of cohomology of \eqref{twisting-SEC-elliptic}, we get  
$$
h^{0}(D)\leq h^{0}(C,\CO_{C}(D))=\chi(D).
$$
Since $h^{2}(D)=0$, by Riemann--Roch theorem on $X$, we have 
$$
h^0(D)=\chi(D)+h^{1}(D)\geq \chi(D).
$$
Hence, $h^0(D)=\chi(D)$ and  $h^{1}(D)=0$.
\end{proof}

From now on, for each integer $1\le k\le10$, we set
$$
C_{k}:=3H-\sum_{1\leq i\leq 10,\\ i\neq k}E_i.
$$
Then, $C_{k}$ is the strict transform of the unique smooth cubic through the nine general points $p_{i}$, $1\leq i\leq 10$ and $i\neq k$.
Thus, $C_{k}$ is an irreducible curve.
Moreover, we have
$C_{k}^{2}=0$, $K_{X}\ldotp C_{k}=0$ and the arithmetic genus $p_{a}(C_{k})=1$.
As a result, $C_{k}\subset X$ is a smooth elliptic curve.

Now we are ready to prove Proposition \ref{10pts-ext2-main}.

\begin{proof}[Proof of Proposition \ref{10pts-ext2-main}]
Based on lemma \ref{10pts-forward-hom-zero-main}, by Riemann--Roch theorem, it suffices to show that 
all $\ext^{2}$ are equal to the Euler characteristic. 
By Serre duality, we have
$$
\ext^{k}(L_{i},L_{j})=h^{2-k}( L_{i}\otimes L_{j}^{-1}\otimes \omega_{X}),
$$
for any $1 \leq i<j\leq 13$ and $k\in \ZN$.
Then,  Lemma \ref{10pts-forward-hom-zero-main} yields that $h^{2}( L_{i}\otimes L_{j}^{-1}\otimes \omega_{X})=0$ for any $1\leq i<j\leq 13$.

Since $K_{X}+A-B_{i}$ intersects negatively with $H$, so we have
$$
\ext^{2}(\CO_{X}(A),\CO_{X}(B_{j}))=0=\chi(B_{j}-A), 
$$
where $3\leq j\leq 10$.
Note that
$\ext^{2}(\CO_{X}(B_{i}),\CO_{X}(B_{j}))=h^{2}(E_{i}-E_{j})=0$ for any $1\leq i<j\leq 10$, and
$\ext^{2}(\CO_{X}(F_{a}),\CO_{X}(F_{a+1}))=\ext^{2}(\CO_{X},\CO_{X}(G))$.
Therefore, it is sufficient to determine the following cases:

(1) Set $\displaystyle D:=K_{X}-A=4H-2E_1-2E_2-\sum_{i=3}^{10}E_i$.
Then, by Riemann--Roch theorem, we have $D\ldotp C_{3}=1=\chi(D)$.
Since $D-C_{3}=H-E_{1}-E_{2}-E_{3}$, so $h^{0}(D-C_{3})=0$, since the points $p_{1}$, $p_{2}$ and $p_{3}$ are in general position.
Thus, by Lemma \ref{cubic-criterion}, we have 
$$
\ext^{2}(\CO_{X},\CO_{X}(A))=h^{0}(K_{X}-A)=1=\chi(\CO_{X},\CO_{X}(A)).
$$
Likewise, since $\displaystyle K_{X}-G+B_{j}=4H-2E_1-2E_j-\sum_{i\notin\{1,j\}}E_{i}$, applying Lemma \ref{cubic-criterion} to $C_{2}$, we get 
$$
\ext^{2}(\CO_{X}(B_{j}),\CO_{X}(G))=h^{0}(K_{X}-G+B_{j})=1=\chi(\CO_{X}(B_{j}),\CO_{X}(G)).
$$

(2) Since $\displaystyle K_{X}-B_{j}=3H-\sum_{i\neq j} E_{i}$, $3\leq j\leq 10$ and there is a unique cubic passing through nine general points,
we obtain
$$
\ext^{2}(\CO_{X},\CO_{X}(B_{j}))=h^{0}(K_{X}-B_{j})=1=\chi(\CO_{X},\CO_{X}(B_{j})), 
$$
where $3\leq j\leq 10$.
Similarly, since $K_{X}-G+A=3H-E_{1}-\sum_{i=3}^{10}E_{i}$,  we have
$$
\ext^{2}(\CO_{X}(A),\CO_{X}(G))=h^{0}(K_{X}-G+A)=1=\chi(\CO_{X}(A),\CO_{X}(G)).
$$

(3) Set $\displaystyle D:=K_{X}-G=10H-4E_{1}-3\sum_{i=2}^{10}E_{i}$. 
By Riemann--Roch theorem, we have $D\ldotp C_{2}=2=\chi(D)$. 
Since $D-C_{2}=7H-3E_{1}-3E_{2}-\sum_{i=3}^{10}2E_{i}$, up to permutations,
by Corollary \ref{10pts-NoSect-case1-cor}, we obtain $h^{0}(D-C_{2})=0$.
It follows from Lemma \ref{cubic-criterion} that
$$
\ext^{2}(\CO_{X},\CO_{X}(G))=h^{0}(K_{X}-G)=2=\chi(\CO_{X},\CO_{X}(G)).
$$

(4) Set $D:=K_{X}-F_{a}=-(4a+3)K_{X}+(a+1)H-aE_{1}-E_{2}$.
It follows from Riemann--Roch theorem that $D\ldotp C_{3}=2a+2=\chi(D)$.
Since
$$
D-C_{3}=(13a+7)H-(5a+2)E_{1}-(4a+3)(E_{2}+E_{3})-\sum_{i=4}^{10}(4a+2)E_{i},
$$
by Corollary \ref{10pts-NoSect-case1-cor}, we get $h^{0}(D-C_{3})=0$.
Then, Lemma \ref{cubic-criterion} yields
$$
\ext^{2}(\CO_{X},\CO_{X}(F_{a}))=h^{0}(K_{X}-F_{a})=2a+2=\chi(\CO_{X},\CO_{X}(F_{a})),
$$
where $a\geq 0$.
In particular, since $\ext^{2}(\CO_{X}(G),\CO_{X}(F_{a}))
=h^{0}(-(4a-1)K_{X}+aH-(a-1)E_{1}-E_{2})$, 
so we get 
$$
\ext^{2}(\CO_{X}(G),\CO_{X}(F_{a}))=h^{0}(K_{X}-F_{a}+G)=2a=\chi(\CO_{X}(G),\CO_{X}(F_{a})),
$$
where $a\geq 0$.

(5) Set $\displaystyle D:=K_{X}-F_{a}+A=-(4a+1)K_{X}+aH-(a-1)E_{1}$.
Note that $D\ldotp C_{2}=2a+1$.
According to Riemann--Roch theorem, $\chi(D)=2a+1=D\ldotp C_{2}$.
If $a=0$, then 
$D-C_{2}=E_{1}-E_{2}$ and thus $h^{0}(D-C_{2})=0$.
If $a\geq 1$, then
$$
D-C_{2}=13aH-(5a-1)E_{1}-(4a+1)E_{2}-\sum_{i=3}^{10} 4aE_{i}.
$$
By Corollary \ref{10pts-NoSect-case1-cor}, we have $h^{0}(D-C_{2})=0$.
Hence, by Lemma \ref{cubic-criterion}, we derive
$$
\ext^{2}(\CO_{X}(A),\CO_{X}(F_{a}))=h^{0}(K_{X}-F_{a}+A)=2a+1=\chi(\CO_{X}(A),\CO_{X}(F_{a})),
$$
where $a\geq 0$.

(6) Set $D:=K_{X}-F_{a}+B_{j}=-(4a+1)K_{X}+(a+1)H-aE_{1}-E_{2}-E_{j}$, $3 \leq j\leq 10$.
For every fixed $j$, let $3\leq k\leq 10$ and $k\neq j$, then $D\ldotp C_{k}=2a+1$.
By Riemann--Roch theorem, we have $D\ldotp C_{k}=2a+1=\chi(D)$.
If $a=0$, then $D-C_{k}=H-E_{2}-E_{j}-E_{k}$. Thus, $h^{0}(D-C_{k})=0$.
If $a\geq 1$, then 
$$
D-C_{k}=(13a+1)H-5aE_{1}-(4a+1)(E_{2}+E_{j}+E_{k})-\sum_{i\notin \{1,2,j,k\}} 4aE_{i}.
$$
Applying Cremona transformation, we get 
$$
h^{0}(D-C_{k})=
h^{0}(13aH-(5a-1)E_{1}-(4a+1)E_{k}-\sum_{i\neq k} 4aE_{i}).
$$
Then, up to permutations, by Proposition \ref{Cremnona-transform-consequ} and Corollary \ref{10pts-NoSect-case1-cor}, we obtain $h^{0}(D-C_{k})=0$.
Finally, by Lemma \ref{cubic-criterion}, we have
$$
\ext^{2}(\CO_{X}(B_{j}),\CO_{X}(F_{a}))=h^{0}(K_{X}-F_{a}+B_{j})=2a+1=\chi(\CO_{X}(B_{j}),\CO_{X}(F_{a})),
$$
where $3\leq j\leq 10$ and $a\geq 0$.
\end{proof}

%-----------------------------------------------------------------------

\subsection{Proof of Proposition \ref{10pts-ext1-main}}

We first obtain the following result:

\begin{prop}\label{10pts-NoSect-case2}
For every integer $a\geq 0$, let $D$ be one of the divisors:
\begin{align*}
& (13a+9)H-(5a+2)E_{1}-\sum_{i=2}^{10}(4a+3)E_{i}, \\
& (13a+16)H-(5a+5)E_{1}-(4a+6)E_{2}-\sum_{i=3}^{10}(4a+5)E_{i}.
\end{align*}
Then $H^{0}(X,\CO_{X}(D))=0$.
\end{prop}

The proof of Proposition \ref{10pts-NoSect-case2} is contained in Appendix \ref{technique-prop2}, which is the same as that of Proposition \ref{10pts-NoSect-case1}. 
Temporarily admitting this proposition, we finish the proof of Proposition \ref{10pts-ext1-main}.

\begin{proof}[Proof of Proposition \ref{10pts-ext1-main}]
For any $1 \leq i<j\leq 13$, by Riemann--Roch theorem, the Euler characteristic   $\chi(L_{j},L_{i}\otimes \omega_{X}^{-1})<0$. Hence, it suffices to show that all the $\Hom$-spaces and $\Ext^{2}$-spaces are trivial.
For $1 \leq i<j\leq 13$,
by Serre duality,  we have
$$
\ext^{2}(L_{j},L_{i}\otimes \omega_{X}^{-1})=h^{0}(L_{j}\otimes L_{i}^{-1}\otimes \CO_{X}(2K_{X})).
$$
By direct computations, we get $\big(L_{j}\otimes L_{i}^{-1}\otimes \CO_{X}(2K_{X})\ldotp\CO_{X}(H) \big)<0$.
Since $H$ is nef, so all $\Ext^{2}$-spaces are trivial.
In the following, we show that all $\Hom$-spaces are trivial.

Firstly, we note that $\displaystyle A-B_{j}-K_{X}=2H-\sum_{i\neq 1,2,j} E_{i}$. Since no seven general points are on the same conic, so we have
$$
\hom(\CO_{X}(B_{j}),\CO_{X}(A-K_{X}))=
h^{0}(A-B_{j}-K_{X})=0.
$$
Note that $B_{i}-K_{X}-B_{j}=C_{j}-E_{i}$.
Consider the exact sequence
$$
\xymatrix@C=0.5cm
{ 
0\ar[r]& \CO_{X}(-E_{i})
\ar[r]& \CO_{X}(C_{j}-E_{i})
\ar[r]& \CO_{C_{j}}(C_{j}-E_{i}) \ar[r]& 0.
}
$$
Since $(C_{j}-E_{i})\ldotp C_{j}=-1$ and $h^{k}(-E_{i})=0$ for $k\in \ZN$, 
the above exact sequence yields that
$$
\hom(\CO_{X}(B_{j}),\CO_{X}(B_{i}-K_{X}))=h^{0}(B_{i}-K_{X}-B_{j})=h^{0}(C_{j}-E_{i})=0.
$$

Secondly, since the divisors
\begin{align*}
-K_{X}-B_{j}&=9H-2E_{j}-\sum_{i\neq j} 3E_{i},
\qquad 
G-K_{X}=16H-6E_{1}-\sum_{i=2}^{10}5E_{i},\\
A-K_{X}-G &=9H-2E_{2}-\sum_{i\neq 2} 3E_{i},\\
-K_{X}-F_{a}&=(13a+16)H-(5a+5)E_{1}-(4a+6)E_{2}-\sum_{i=3}^{10} (4a+5)E_{i}, \\
G-K_{X}-F_{a}&=(13a+3)H-5aE_{1}-(4a+2)E_{2}-\sum_{i=3}^{10} (4a+1)E_{i}, \\
A-K_{X}-F_{a}&=(13a+9)H-(5a+2)E_{1}-\sum_{i=2}^{10}(4a+3)E_{i},
\end{align*}
by Proposition \ref{10pts-NoSect-case2}, 
we have 
\begin{align*}
\hom(\CO_{X}(B_{j}),\CO_{X}(-K_{X}))&= h^{0}(-K_{X}-B_{j}) =0, \\
\hom(\CO_{X}(G),\CO_{X}(-K_{X}))& =h^{0}(-K_{X}-G)=0,\\
\hom(\CO_{X}(G),\CO_{X}(A-K_{X}))&=h^{0}(A-K_{X}-G)=0,\\
\hom(\CO_{X}(F_{a}),\CO_{X}(-K_{X}))&=
h^{0}(-K_{X}-F_{a})=0,\\
\hom(\CO_{X}(F_{a}),\CO_{X}(A-K_{X}))&= h^{0}(A-K_{X}-F_{a})=0,\\
\hom(\CO_{X}(F_{a}),\CO_{X}(G-K_{X}))&=h^{0}(G-K_{X}-F_{a})=0.
\end{align*}

Additionally, since $\displaystyle -A-K_{X}=10H-4E_{1}-4E_{2}-\sum_{i=3}^{10}3E_{i}$, 
by Proposition \ref{Cremnona-transform-consequ} and Proposition \ref{10pts-NoSect-case2}, we obtain
$$
h^{0}(10H-4E_{1}-4E_{2}-\sum_{i=3}^{10}3E_{i})=h^{0}(9H-3E_{1}-3E_{2}-2E_{2}-\sum_{i=4}^{10}3E_{i})=0.
$$
Thus, we have
$$
\hom(\CO_{X}(A),\CO_{X}(-K_{X}))=h^{0}(-A-K_{X})=0.
$$
Similarly, since $B_{j}-G-K_{X}=-3K_{X}+H-E_{1}-E_{j}$, we obtain
$$
\hom(\CO_{X}(G),\CO_{X}(B_{j}-K_{X}))=h^{0}(B_{j}-G-K_{X})=0.
$$

Finally, since
$
B_{j}-K_{X}-F_{a}=
(13a+10)H-(5a+3)E_{1}-(4a+4)(E_{2}+E_{j})-\sum_{i\neq 1,2,j}^{10}(4a+3)E_{i},
$
by Proposition \ref{Cremnona-transform-consequ}, we have
$$
h^{0}(B_{j}-K_{X}-F_{a})
=h^{0}((13a+9)H-(5a+2)E_{1}-(4a+3)(E_{2}+E_{j})-\sum_{i\neq 1,2,j}^{10}(4a+3)E_{i}).
$$
Thus, by Proposition \ref{10pts-NoSect-case2}, we get
$$
\hom(\CO_{X}(F_{a}),\CO_{X}(B_{j}-K_{X}))=h^{0}(B_{j}-K_{X}-F_{a})=0.
$$
This completes the proof. \end{proof}

\subsection{Proof of Theorem \ref{distinct-phantoms-thm}}

In the setting of  Theorem \ref{mainthm}, 
for each integer $a\geq 0$, we set the exceptional collection
\begin{align*}
\mathbb{A}_{a}&=
\{L_{1},L_{2},L_{3},\cdots,L_{10},L_{11},L_{12},L_{13}\}\\
& = \{\CO_{X},\CO_{X}(A),\CO_{X}(B_{3}),\cdots,\CO_{X}(B_{10}),\CO_{X}(G),\CO_{X}(F_{a}),\CO_{X}(F_{a+1})\}.
\end{align*}
We use $\langle \mathbb{A}_{a} \rangle$ to denote the full triangulated subcategory of $\DC(X)$ generated by the exceptional collection $\mathbb{A}_{a}$ and $\RA_{X}^{(a)}$ its right orthogonal complement.

\begin{prop}\label{HH-A-phantom}
For every integer $a\geq 0$, 
the dimensions of the Hochschild cohomology groups of $\RA_{X}^{a}$ are given as follows:
\begin{align*}
\dim \HH^{0}(\RA_{X}^{(a)})&=1, & \dim \HH^{1}(\RA_{X}^{(a)})&=0, \\
\dim \HH^{2}(\RA_{X}^{(a)})&=12, & \dim \HH^{3}(\RA_{X}^{(a)})&=88a^{2}+220a+206, \\
\dim \HH^{4}(\RA_{X}^{(a)})& = 248a^{2}+620a+373, & \dim \HH^{5}(\RA_{X}^{(a)})&=232a^{2}+580a+180,\\
\dim \HH^{6}(\RA_{X}^{(a)})&= 72a^{2}+180a, & \dim \HH^{i}(\RA_{X}^{(a)})&=0, \, i\geq 7.
\end{align*}
In particular, the height $h(\mathbb{A}_{a})=4$.
\end{prop}

\begin{proof}
Recall that the anticanonical pseudoheight is
$$
\phac(\mathbb{A}_{a})=\min_{1\leq a_{0}<\cdots<a_{p}\leq 13}\Big(\sum_{i=0}^{p-1} e(L_{a_{i}},L_{a_{i+1}})+e(L_{a_{p}},L_{a_{0}}\otimes \omega_{X}^{-1})-p\Big).
$$
If $p=0$, then by Remark \ref{no-section-anticanonical}, the relative height 
$$
e(L_{a_{0}},L_{a_{0}}\otimes \omega_{X}^{-1})=e(\CO_{X}, \CO_{X}(-K_{X}))=+\infty.
$$
Suppose $p\geq 1$.  By Proposition \ref{10pts-ext2-main}, we have  $e(L_{a_{i}},L_{a_{i+1}})=2$ or $+\infty$ for $0 \leq i\leq p-1$. Moreover, by Proposition \ref{10pts-ext1-main}, we have  $e(L_{a_{p}},L_{a_{0}}\otimes \omega_{X}^{-1})=1$.
It follows that
$$
\sum_{i=0}^{p-1} e(L_{a_{i}},L_{a_{i+1}})+e(L_{a_{p}},L_{a_{0}}\otimes \omega_{X}^{-1})-p=p+1 \textrm{ or }+\infty.
$$
Hence, the anticanonical pseudoheight $\phac(\mathbb{A}_{a})=2$ and thus the pseudoheight $$
\ph(\mathbb{A}_{a})=\phac(\mathbb{A}_{a})+\dim X=4.
$$
In particular, the height $h(\mathbb{A}_{a})\geq \ph(\mathbb{A}_{a})=4$.
Thus, the normal Hochschild
cohomology $\NHH^{k}(\langle \mathbb{A}_{a} \rangle,X)=0$ for $k\leq 3$.

By Theorem \ref{thm:nhh_triang}, 
there is a long exact sequence
$$
\xymatrix@C=0.3cm{
\cdots \ar[r] & \NHH^{k}(\langle \mathbb{A}_{a} \rangle,X) \ar[r] & \HH^{k}(X) \ar[r] &  \HH^{k}(\RA_{X}^{(a)})\ar[r] & \NHH^{k+1}(\langle \mathbb{A}_{a} \rangle,X) \ar[r] & \cdots.
}
$$
Since $\HH^{0}(X)=\CN$, $\HH^{1}(X)=0$, $\HH^{2}(X)=\CN^{12}$ and $\HH^{k}(X)=0$ for $k\geq 3$,
we obtain $\HH^{2}(\RA_{X}^{(a)})\cong \HH^{2}(X)=\CN^{12}$ and 
$$
\HH^{i}(\RA_{X}^{(a)})\cong \NHH^{i+1}(\langle \mathbb{A}_{a} \rangle,X)
$$ 
for $i\geq 3$.

Next, we will determine the normal Hochschild
cohomology $\NHH^{j}(\langle \mathbb{A}_{a} \rangle,X)$ for $j\geq 4$.
By Proposition \ref{NHH--spectral-prop}, we have Kuznetsov's spectral sequence
\begin{equation}\label{NHH-spectral-seq-10pits-1}
E_{1}^{-p,q} \Rightarrow \NHH^{q-p}(\langle \mathbb{A}_{a} \rangle,X),
\end{equation}
where $E_{1}^{-p,q}$ is defined by 
\begin{equation*}
\bigoplus_{\substack{1 \leq a_{0} < \cdots < a_{p} \leq 13 \\ k_{0} + \cdots + k_{p} = q}} \Ext^{k_{0}}(L_{a_{0}}, L_{a_{1}}) \otimes \cdots \otimes 
\Ext^{k_{p-1}}(L_{a_{p-1}}, L_{a_{p}}) \otimes 
 \Ext^{k_{p}}(L_{a_{p}}, \mathcal{S}^{-1}(L_{a_{0}})),
\end{equation*}
where $\mathcal{S}(-):=-\otimes \omega_{X}[2]$ is the Serre functor of $\DC(X)$.
By Proposition \ref{10pts-ext2-main}, we derive that the terms $\Ext^{k_{i}}(L_{a_{i}},L_{a_{i+1}})$ have non-trivial values only for $k_{i}=2$.
According to Proposition \ref{10pts-ext1-main}, we know that 
$$
\RHom(L_{a_{p}},\mathcal{S}^{-1}(L_{a_{0}}))=\CN^{-\chi(\mathcal{S}^{-1}(L_{a_{0}})\otimes L_{a_{p}}^{-1})}[-3].
$$
Hence, $E_{1}^{-p,q}$ is nontrivial only if $2p+3=q$. 
Consequently, the spectral sequence \eqref{NHH-spectral-seq-10pits-1} degenerates on the $E_{1}$-page.
Moreover, if $q-p=j\geq 4$ and $2p+3=q$, then $p=j-3$ and $q=2j-3$.
This yields that
$$
\NHH^{j}(\langle \mathbb{A}_{a} \rangle,X) \cong E_{1}^{-j+3,2j-3},
$$
where $j\geq 4$.
Based on Proposition \ref{10pts-ext2-main} and Proposition \ref{10pts-ext1-main}, by a straightforward computation, we have 
\begin{align*}
E_{1}^{-1,5}&=\CN^{88a^{2}+220a+206},& E_{1}^{-2,7}&=\CN^{248a^{2}+620a+373}, \\
E_{1}^{-3,9}&=\CN^{232a^{2}+580a+180}, &
E_{1}^{-4,11}&=\CN^{72a^{2}+180a},
\end{align*}
and $E_{1}^{-j+3,2j-3}=0$ for $j\geq 7$.
This concludes the proof.
\end{proof}

\begin{rem}\label{no-section-anticanonical}
For any $i\geq 0$, we have $\Ext^{i}(\CO_{X},\CO_{X}(-K_{X}))=0$. 
In particular, the relative height $e(\CO_{X},\CO_{X}(-K_{X}))=+\infty$.
In fact, for any $i\geq 0$, $\ext^{i}(\CO_{X},\CO_{X}(-K_{X}))= h^{i}(-K_{X})$. 
Note that by Riemann--Roch theorem, $\chi(-K_{X})=0$. 
Since $K_{X}\ldotp H=-3<0$, by Serre duality, we get $h^{2}(-K_{X})=h^{0}(2K_{X})=0$. 
Thus, it remains to show that $h^{0}(-K_{X})=0$. 
This is equivalent to the non-existence of a cubic curve passing through the ten general points, which holds by the generality assumption on the points.
\end{rem}

Now we are in the position to finish the proof of Theorem \ref{distinct-phantoms-thm}.

\begin{proof}[Proof of Theorem \ref{distinct-phantoms-thm}]
By \cite[Remark 5.5]{Kra24}, the third Hochschild cohomology of Krah's phantom is of dimension $446$.
Note that the third Hochschild cohomology of the phantom in \cite[Theorem 1.2]{KKL+26} is of dimension $998$, and the third Hochschild cohomology of the phantom in \cite[Theorem 3.7]{KKL+26} is of dimension $446$; see Remark \ref{pf-thm1.2-rem}.
Moreover, 
for the phantom in \cite[Theorem 1.1]{KKL+26} and \cite[Theorem 1.1]{MXY25}, the second Hochschild cohomology group has dimension $\geq 14$ (\cite[Proposition 3.5]{KKL+26}); in fact, it is of dimension $27$ (\cite[Proposition A.2]{MXY25}).
As a result, Proposition \ref{HH-A-phantom} concludes the proof.
\end{proof}
  
\begin{rem}\label{pf-thm1.2-rem}
With computer-assisted computations in Macaulay2 as explained in \cite[Appendix B]{KKL+26}, 
by direct computations as in the proof of Proposition \ref{HH-A-phantom}, we obtain that the third Hochschild cohomology of the phantoms in \cite[Theorem 1.3]{KKL+26} and \cite[Theorem 3.7]{KKL+26} have dimensions $998$ and $446$, respectively.
In fact, the Hochschild cohomology groups of the phantom in \cite[Theorem 3.7]{KKL+26} have the same dimensions as that of Krah's phantom. 
At present, it is unknown  whether or not the phantom in \cite[Theorem 3.7]{KKL+26} is equivalent to Krah's phantom.
Besides, it is worth noting that Dumnicki's diagram-cutting method holds for the blow-up of $\mathbf{F}_{2}$ at $9$ points in general position.
Hence, Dumnicki's diagram-cutting method can also be applied to study both the case of the blow-up of $\mathbf{F}_{2}$ at nine points and the case of the blow-up of $\mathbb{P}^{2}$ at ten points.
\end{rem}

%===========================================================================

\section{Applications}\label{appl-t-struture}

In this paper, we prove that all phantom categories constructed in Theorem \ref{mainthm} admit bounded $t$-structures.
We also construct some explicit objects in the corresponding hearts.  

\subsection{Existence of bounded $t$-structures}

We begin by recalling the definition of a bounded $t$-structure.

\begin{defn}
Let $\mathcal{D}$ be a triangulated category. A {\it $t$-structure} on $\mathcal{D}$ is a pair of full subcategories
$\tau=(\mathcal{D}^{\leq 0},\mathcal{D}^{\geq 0})$
such that the following conditions hold:
\begin{enumerate}
\item[(i)] $\mathcal{D}^{\leq 0}[1]\subseteq \mathcal{D}^{\leq 1}$ and $\mathcal{D}^{\geq 0}[-1]\subseteq \mathcal{D}^{\geq -1}$;
\item[(ii)] For every $X\in \mathcal{D}^{\leq 0}$ and $Y\in \mathcal{D}^{\geq 1}$, $\Hom(X,Y)=0$,
\item[(iii)] For every object $E\in \mathcal{D}$, there is an exact triangle
$$
\xymatrix@C=0.5cm{
A \ar[r]& E \ar[r] & B \ar[r]& A[1]
}
$$
with
$ A\in \mathcal{D}^{\leq 0}$ and
$B\in \mathcal{D}^{\geq 1}$,
\end{enumerate}
where $\mathcal{D}^{\leq n}:=\mathcal{D}^{\leq 0}[-n]$ and  $\mathcal{D}^{\geq n}:=\mathcal{D}^{\geq 0}[-n]$.
In addition, the subcategories $\mathcal{D}^{\leq 0}$ and $\mathcal{D}^{\geq 0}$ of $\mathcal{D}$ are called the {\it connective} and {\it coconnective} parts of the
$t$-structure $\tau$.
\end{defn} 

Let $\tau=(\mathcal{D}^{\leq 0},\mathcal{D}^{\geq 0})$ be a $t$-structure on $\mathcal{D}$. 
Then, the {\it heart} of the $t$-structure $\tau$ is defined to be the full subcategory
$$
{}^{\tau}\mathcal{D}^{\heartsuit}:=\mathcal{D}^{\leq 0}\cap \mathcal{D}^{\geq 0} \subset \mathcal{D}.
$$
It is an abelian category.
Moreover, the $t$-structure $\tau$ is called {\it bounded} if
$$
\mathcal{D}
=
\bigcup_{n\in\mathbb{Z}}\mathcal{D}^{\leq n}
=
\bigcup_{n\in\mathbb{Z}}\mathcal{D}^{\geq n}.
$$
Equivalently, for every object $E\in \mathcal{D}$, there exist integers
$a\leq b$ such that $E\in \mathcal{D}^{\geq a}\cap \mathcal{D}^{\leq b}$.
A typical example is the derived category of an abelian category. 
For example, let $X$ be a smooth proper variety. 
Then, the derived category $\DC(X)$ has a standard bounded $t$-structure $\tau_{X}$:
$$
\mathcal{D}^{\leq 0}:=\{E\in \DC(X) \mid \mathcal{H}^{i}(E)=0 \textrm{ for }\, i>0 \}
$$
$$
\mathcal{D}^{\geq 0}:=\{E\in \DC(X) \mid \mathcal{H}^{i}(E)=0 \textrm{ for }\, i<0 \}
$$
Its heart ${}^{\tau_{X}}\DC(X)^{\heartsuit}$ is naturally identified with the abelian category ${\rm Coh}(X)$ of coherent sheaves. 
However, for the semi-orthogonal components of $\DC(X)$, the following question is still widely open.

\begin{quest}[{\cite[Question 1.1]{KLP26}}]
Let $X$ be a smooth proper variety.   
If $\mathscr{C}\subset \DC(X)$ is an admissible subcategory, then does $\mathscr{C}$ admit a bounded $t$-structure?
\end{quest}

In order to construct bounded $t$-structures on admissible subcategories, 
Kuznetsov--Liu--Perry \cite{KLP26} introduced the notion of  connectively induced $t$-structure.
Let us recall its definition:

\begin{defn}[{\cite[Definition 5.1 (a)]{KLP26}}]
Let $\mathcal{D}$ be a triangulated category with a $t$-structure $\tau = (\mathcal{D}^{\leq 0}, \mathcal{D}^{\geq 0})$ on $\mathcal{D}$.
Suppose $\mathscr{C}\subset \mathcal{D}$ is a full triangulated subcategory. 
We say that $\tau$ {\it connectively induces a $t$-structure} on $\mathscr{C}$ if the pair $\tau^{-}_{\mathscr{C}}=(\mathscr{C}^{\leq 0}, \mathscr{C}^{\geq 0})$ with
\begin{align*}
 \mathscr{C}^{\leq 0}& := {}^{\tau}\mathcal{D}^{\leq 0} \cap \mathscr{C} \\
 \mathscr{C}^{\geq 0}&:=\{E\in\mathscr{C} \mid \Hom(F, E)=0 \textrm{ for all }\; F\in \mathscr{C}^{\leq 0}[1] \}
\end{align*}
defines a $t$-structure on $\mathscr{C}$. We also say that $\tau^{-}_{\mathscr{C}}$ is the {\it connectively induced  $t$-structure} on $\mathscr{C}$ and we denote by ${}^{\tau_{\mathscr{C}}^{-}}\mathscr{C}^{\heartsuit}$ the heart of $t$-structure $\tau^{-}_{\mathscr{C}}$.
\end{defn}

This paper is mainly interested in the admissible subcategory orthogonal to a given exceptional collection. 
Recently, Kuznetsov--Liu--Perry \cite{KLP26} established the following theorem, which induces a $t$-structure on the orthogonal complement of an exceptional collection from a $t$-structure on the ambient triangulated category:  

\begin{thm}[{\cite[Theorem 1.7]{KLP26}}]\label{KLP-thm}
Let $\mathcal{D}$ be a $\CN$-linear triangulated category and $\tau = (\mathcal{D}^{\leq 0}, \mathcal{D}^{\geq 0})$ a noetherian $t$-structure on $\mathcal{D}$.
Suppose there is a semi-orthogonal decomposition
$$
\mathcal{D}=\langle L_{1}, L_{2}, \cdots, L_{m}, \mathscr{C} \rangle
$$
where $L_{1},L_{2}, \cdots, L_{m}\in {}^{\tau}\mathcal{D}^{\heartsuit}$ is an exceptional collection such that $\Hom(L_{i}, L_{j})=0$
for all $1\leq i<j\leq m$.
Then $\tau$ connectively induces a $t$-structure $\tau^{-1}_\mathscr{C}$ on $\mathscr{C}$, which is bounded if $\tau$ is bounded.
\end{thm}

This theorem enables us to establish the existence of bounded $t$-structures on our phantom categories.

\begin{thm}[Theorem \ref{phantom-bounded-t-structure}]\label{bd-t-str-ourphant}
Let $\mathscr{A}_{X}^{(a)}$ be a phantom category in Theorem \ref{mainthm}, where $a\geq 0$.
Then $\mathscr{A}_{X}^{(a)}$ has a bounded $t$-structure. 
\end{thm}

\begin{proof}
Let $X$ be the blow-up of $\PB^{2}$ at $10$ points in general position.
Note that in Theorem \ref{mainthm}, for each $a\geq 0$, the phantom $\mathscr{A}_{X}^{(a)}$ is right orthogonal to a non-full exceptional collection \eqref{EC-sequ-10pts-main} of line bundles of length $13$.
For simplicity, let $\mathscr{C}_{X}$ be the left orthogonal complement of the exceptional collection \eqref{EC-sequ-10pts-main}.
Then, there is an equivalence 
\begin{equation}\label{equiv-two-phs}
-\otimes \omega_{X}: \mathscr{C}_{X}\rightarrow \mathscr{A}_{X}^{(a)}.
\end{equation}
It is known that a bounded $t$-structure on $\mathscr{C}_{X}$ induces a bounded $t$-structure on $\mathscr{A}_{X}^{(a)}$ via the equivalence \eqref{equiv-two-phs}.
Consequently, it suffices to establish a bounded $t$-structure on $\mathscr{C}_{X}$.
Now consider the standard bounded $t$-structure $\tau_{X}$ on $\DC(X)$. By Theorem \ref{KLP-thm}, to show that $\tau_{X}$ connectively induces a bounded $t$-structure
$(\tau_{X})^{-}_{\mathscr{C}_{X}}$ on $\mathscr{C}_{X}$, it is sufficient to show  that all forward $\Hom$-spaces of the exceptional collection are trivial. In fact, the required forward vanishings follow from Lemma \ref{10pts-forward-hom-zero-main}.
\end{proof}

\begin{rem}
In \cite[Theorem 2.3.3]{Liu25}, Liu showed that every geometric phantom category does not admit noetherian bounded $t$-structures. 
As a result, the bounded $t$-structures constructed in Theorem \ref{bd-t-str-ourphant} are not noetherian.   
\end{rem}

On the other hand, the examples of co-connective DG algebras whose derived categories are phantom categories were first constructed by Mattoo \cite{Mat25} and subsequently by the authors \cite{MXY25}, thereby answering a question of Ben Antieau. 
Recently, using the existence of a bounded $t$-structure, Kuznetsov--Liu--Perry \cite[Corollary 6.14]{KLP26} provided more examples answering the same question.
We give furthermore examples as follows.

\begin{cor}
Let $\mathscr{A}_{X}^{(a)}$ be a phantom category in Theorem \ref{mainthm}, where $a\geq 0$.
Then there
exists a classical generator $T\in  \mathscr{A}_{X}^{(a)}$ such 
that $\Ext^{i}(T,T) =0$ for $i<0$. In particular, $\RHom(T,T)$ is a co-connective DG algebra
such that its derived category is a phantom category.  
\end{cor}

\begin{proof}
Based on Theorem \ref{bd-t-str-ourphant}, the proof follows that of \cite[Corollary 6.14]{KLP26}.
In fact, by the same proof as \cite[Corollary 6.14]{KLP26}, there is a classical generator $G\in \mathscr{C}_{X}$ such 
that $\Ext^{i}(G,G) =0$ for $i<0$.
Thus, by the equivalence \eqref{equiv-two-phs}, the object 
$$
T:=G\otimes \omega_{X}\in \mathscr{A}_{X}^{(a)}
$$
is a classical generator satisfying $\Ext^{i}(T,T) =0$ for $i<0$.
\end{proof}

\subsection{Explicit objects in the heart}

In this subsection, we describe explicit objects in the heart of the bounded $t$-structures induced on the phantom categories in Theorem \ref{mainthm}. 
We adopt the strategies in \cite[\S 5.1]{Mat25} and \cite[Example 6.17]{KLP26}. 

\begin{setup}\label{more-ex-setup}
Let $X$ be the blow-up of $\PB^{2}$ at $n$ points in general position, where $n=10$ or $n=11$.
Let 
$$
\{L_{1}, L_{2}, L_{3},\cdots,L_{n}, L_{n+1}, L_{n+2}, L_{n+3}\}
$$
be a non-full exceptional collection of line bundles on $X$, satisfying:
\begin{enumerate}
\item[(i)] $L_{n+2}=\CO_{X}(-M)$, where $M$ is an ample divisor;
\item[(ii)] $\Ext^{k}(L_{i},L_{j})=0$ for $k\neq 2$ and $1 \leq i<j\leq n+3$;
\item[(iii)]  $\Ext^{2}(L_{i},L_{j})=0$ for $3 \leq i<j\leq n$; moreover, in all remaining cases, $\Ext^{2}(L_{i},L_{j})\neq 0$;
\item[(iv)] $L_{n+2}\ldotp L_{i}<0$ for all $i\neq n+2$.
\end{enumerate}
\end{setup}

Then, there is a semi-orthogonal decomposition
\begin{align*}
\DC(X) &=\langle \mathscr{A}_{X}, L_{1}, L_{2}, L_{3}, \cdots, L_{n}, L_{n+1}, L_{n+2}, L_{n+3}\rangle \\
&=\langle L_{1}, L_{2}, L_{3}, \cdots, L_{n}, L_{n+1}, L_{n+2}, L_{n+3}, \mathscr{C}_{X} \rangle.
\end{align*}
Thus, the admissible subcategories  $\mathscr{A}_{X}$ and $\mathscr{C}_{X}$ are phantom categories. 
Under the condition (ii), by Theorem \ref{KLP-thm}, the standard bounded $t$-structure $\tau_{X}$ connectively induces a bounded $t$-structure $(\tau_{X})^{-}_{\mathscr{C}_{X}}$ on $\mathscr{C}_{X}$.
Similar to Theorem \ref{bd-t-str-ourphant}, 
the bounded $t$-structure $(\tau_{X})^{-}_{\mathscr{C}_{X}}$ induces a bounded $t$-structure on $\mathscr{A}_{X}$ under the equivalence $-\otimes \omega_{X}: \mathscr{C}_{X}\rightarrow \mathscr{A}_{X}$.

In the Set-up \ref{more-ex-setup}, 
since the divisor $M$ is ample, for $m \gg 0$, there is a smooth curve $C\in |mM|$ with the genus $g$. 
Let $\jmath: C\hookrightarrow X$ be the inclusion and let $L\in \Pic^{g-1}(X)$ be
generic. 
Recall that a generic line bundle on $C$ of degree $d<g$ has no non-trivial global sections. 
In fact, the image of the Abel--Jacobi map
 $$
 \begin{array}{cccl}
 & {\rm Sym}^{d}(C)  & \longrightarrow  & \mathrm{Pic}^{d}(C)  \\
&(p_{1}+\cdots+p_{d}) &\longmapsto &
 \CO_{C}(p_{1}+\cdots+p_{d}),
\end{array}
$$
consists precisely of the effective line bundles of degree $d$.
Since the dimension
$$
\dim {\rm Sym}^{d}(C) =d<g=\dim \mathrm{Pic}^{d}(C),
$$
so the image of the Abel--Jacobi map is a proper closed subvariety. 
Therefore, a generic line bundle of degree $d<g$ has no non-trivial global sections.

Similar to \cite[Lemma 5.9]{Mat25}, we have:
\begin{lem}
Let $\mathcal{L}:=\jmath_{\ast}L$.
Then, for any $1 \leq i\leq n+2$,   $\RHom(L_{i}, \mathcal{L})$ is concentrated in degree $1$. 
\end{lem}

\begin{proof}
For every divisor $D$ on $X$, 
there is an isomorphism
\begin{align*}
\Ext^{\ast}(\mathcal{O}_X(D),\mathcal{L})
& =\Ext^{\ast}(\mathcal{O}_X(D),\jmath_{\ast}L) \cong H^{\ast}(C,L\otimes \CO_{C}(-C\ldotp D)) \\
&= H^{\ast}(C,L(mM \ldotp D)),    
\end{align*}
where $L(mM \cdot D):=L\otimes \CO_{C}(mM \ldotp D)$ is a line bundle of degree $g-1+mM \ldotp D$.
By (iv) in Set-up \ref{more-ex-setup}, for each $i\neq 12$, we have $M\ldotp L_{i}<0$.
Set $\CO_{X}(D):=L_{i}$ and 
then $L(mM \ldotp D)$ is a generic line bundle of degree strictly less than $g$. 
Therefore, we have
$$ 
H^{0}(C,L(mM \ldotp D))=0.
$$
Thus, by the Riemann--Roch theorem, the lemma follows.
\end{proof}

Consider the semi-orthogonal decomposition 
$$
\DC(X)=\langle \mathscr{P}_{X},L_{n+3}^{-1},\cdots,L_{1}^{-1}\rangle.
$$
Let $\imath:\mathscr{P}_{X}\rightarrow \DC(X)$ be the inclusion functor and $\imath^{\ast}: \DC(X) \rightarrow \mathscr{P}_{X}$ the left adjoint of the inclusion functor $\imath$.
Set 
$$
\mathcal{J}:=\imath^{\ast}\mathcal{L}=\imath^{\ast}\jmath_{\ast}L \in \mathscr{P}_{X}.
$$

\begin{lem}\label{J-descrip-n-pts}
The object $\mathcal{J}$ is concentrated in degrees $[0,5]$.
Moreover,
$\mathcal{H}^{0}(\mathcal{J})$ fits into the short exact sequence
\begin{equation*}
\xymatrix@C=0.5cm{
0\ar[r]& \mathcal{L}  \ar[r]& \mathcal{H}^{0}(\mathcal{J}) \ar[r]&  \displaystyle \bigoplus_{i=1}^{n+3}
L_{i}^{-1}\otimes \Ext^{1}(L_{i}^{-1},\mathcal{L}) \ar[r]& 0,
}
\end{equation*}
and for each $1\leq r\leq 5$ and a tuple $\boldsymbol{i}=(i_{0},\cdots,i_{r})$, 
$$
\mathcal{H}^{r}(J)=\bigoplus_{\boldsymbol{i}\in \mathcal{I}_{r}} L_{i_{0}}^{-1}
\otimes
\Big(
\bigotimes_{k=0}^{r-1}
\Ext^{2}\left(L_{i_{k}}^{-1},L_{i_{k+1}}^{-1}\right)
\Big)
\otimes
\Ext^{1}(L_{i_{r}}^{-1},\mathcal{L}), 
$$
where the index sets are given as follows:
\begin{align*}
\mathcal{I}_{1}
:=& \;
\{(i_{0},i_{1}) \mid 
2\leq i_{0}\leq n+3, 
1\leq i_{1}<i_{0}, 
i_{1}\leq 2 \textrm{ or }  i_{0} \geq n+1\},
\\
\mathcal{I}_{2}
:= &\;
\{(i_{0},i_{1},i_{2})\mid n+1\leq i_{0}\leq n+3,\ n+1\leq i_{1}<i_{0},\ 1\leq i_{2}<i_{1}\} \\
\cup &\; 
\{(i_{0},i_{1},i_{2})\mid n+2\leq i_{0}\leq n+3,\ 3\leq i_{1}\leq n,\ 1\leq i_{2}\leq 2\} 
\\ 
\cup&\;  
\{(i_{0},2,1)\mid 3\leq i_{0}\leq n+3\},
\\
\mathcal{I}_{3}
:= & \{(n+3,n+2,n+1,i_{3})\mid 1\leq i_{3}\leq n \} \\
\cup &\; 
\{(n+3,n+1,i_{2},i_{3})\mid 2\leq i_{2}\leq n,\ 1\leq i_{3}<i_{2} \} \\
\cup &\;
\{(i_{0},i_{1},2,1)\mid n+1\leq i_{0}\leq n+3,\ 3\leq i_{1}<i_{0} \}  \\
\cup &\; 
\{(i_{0},i_{1},i_{2},i_{3})\mid n+1\leq i_{0}\leq n+3,\ i_{1}=i_{0}-1, 3\leq i_{2}\leq n,\ 1\leq i_{3}\leq 2 \},
\\
\mathcal I_{4}
:= &\; 
\{(i_{0},i_{1},i_{2},2,1)\mid 
n+2\leq i_{0}\leq n+3,\ 
n+1\leq i_{1}<i_{0},\ 
3\leq i_{3}\leq n
\} \\
\cup & \; 
\{(n+3,n+2,n+1,i_{3},i_{4})\mid 
3\leq i_{3}\leq n,\ 
1\leq i_{4}\leq 2
\} \\
 &\; 
\{(n+3,n+2,n+1,2,1)\},
\\
\mathcal{I}_{5}:= & \;\{(n+3,n+2,n+1,i_{3},2,1)\mid 3\leq i_{3}\leq n\}.  
\end{align*}
\end{lem}

\begin{proof}
The idea is to apply \cite[Proposition 2.17]{Mat25} to calculate the object $\imath^{\ast}\mathcal{L}$.
Under the Set-up \ref{more-ex-setup},
the argument follows exactly the same lines as the proof of \cite[Proposition 5.10]{Mat25}.    
\end{proof}

In particular, following \cite[Example 6.17]{KLP26}, 
we obtain some explicit objects in the heart of an induced bounded $t$-structure on $\mathscr{C}_{X}$.
Note that there is an equivalence between two phantoms
$$
\RCH(-,\CO_{X}): \mathscr{P}_{X} \longrightarrow \mathscr{C}_{X}.
$$

\begin{prop}\label{uniform-example}
Let $\mathcal{J}^{\prime}
:=\RCH(\mathcal{J},\CO_{X})[1]\in \mathscr{C}_{X}$. 
Then the object 
$\mathcal{J}^{\prime}\in ^{(\tau_{X})_{\mathscr{C}_{X}}^{-}}\mathscr{C}_{X}^{\heartsuit}$.
\end{prop}

\begin{proof}
By Lemma \ref{J-descrip-n-pts} and the spectral sequence
\begin{equation*}
E_{2}^{p,q} 
=\CExt^p(\mathcal{H}^{-q}(\mathcal{J}), \CO_{X}) \Rightarrow \CExt^{p+q}(\mathcal{J}, \CO_{X}),
\end{equation*}
it follows that  $\mathcal{J}^{\prime}$  is concentrated in degrees $[-6,0]$.
By Grothendieck--Verdier duality and projection formula, we have
\begin{align*}
\RCH(\mathcal{L},\CO_{X})= &\; \RCH(\jmath_{\ast} L,\CO_{X})\cong \jmath_{\ast}\RCH(L, \omega_{C}\otimes\omega_{X}^{-1})[-1] \nonumber \\
\cong & \;  \jmath_{\ast}\RCH(L, \CO_{C}(C))[-1] \cong \jmath_{\ast}(L^{-1}\otimes \CO_{C}(C))[-1] \nonumber \\
\cong &\; \jmath_{\ast}L^{-1} \otimes \CO_{X}(C)[-1]. 
\end{align*}
Moreover, by direct computations,
the $0$-th cohomology sheaf
$$
\mathcal{H}^{0}(\mathcal{J}^{\prime})
\cong 
\CExt^{1}(\mathcal{H}^{0}(\mathcal{J}),\CO_{X})
\cong
\CExt^{1}(\mathcal{L},\CO_{X})
\cong
\jmath_{\ast} L^{-1}\otimes \CO_{X}(C),
$$ 
$\mathcal{H}^{-k}(\mathcal{J}^{\prime})$ is a direct sum of copies of $L_{k}, \cdots, L_{n+3}$ for $1\leq k\leq 3$, and 
$\mathcal{H}^{-k}(\mathcal{J}^{\prime})$ is a direct sum of copies of $L_{n+k-3},\cdots, L_{n+3}$ for $4\leq k\leq 6$.
This yields that the object $\mathcal{J}^{\prime}$ satisfies the conditions of \cite[Lemma 5.8(a)]{KLP26} and thus $\mathcal{J}^{\prime}\in ^{(\tau_{X})_{\mathscr{C}_{X}}^{-}}\mathscr{C}_{X}^{\heartsuit}$.
\end{proof}

As an application, we have:

\begin{ex}\label{10pts-object-in-heart}
Let $X$ be the blow-up of $\PB^{2}$ at $10$ points in general position. 
For every integer $a\geq 0$, twisting the non-full exceptional collection \eqref{NEC-10pts-case} by $\CO_{X}(-aE_{1}-E_{2})$, we get a non-full exceptional collection of line bundles
\begin{align}\label{twisted-10pts-EC}
& \{\CO_{X}(-aE_{1}-E_{2}),\CO_{X}(A-aE_{1}-E_{2}),
\begin{smallmatrix} \CO_{X}(B_{3}-aE_{1}-E_{2}) \\ \cdots \\ \CO_{X}(B_{10}-aE_{i}-E_{2})\end{smallmatrix}, \CO_{X}(G-aE_{1}-E_{2}), \nonumber \\
 & \;\;\;\; \CO_{X}(F_{a}-aE_{1}-E_{2}),\CO_{X}(F_{a+1}-aE_{1}-E_{2})\}.
\end{align}
Then, its right orthogonal complement is the phantom category 
$$
\RA_{X}^{(a)} \otimes \CO_{X}(-aE_{1}-E_{2}):=\{ \mathcal{F} \otimes \CO_{X}(-aE_{1}-E_{2}) \mid \mathcal{F}\in \RA_{X}^{(a)} \}.
$$
Following the proof of \cite[Proposition 5.1]{Mat25}, for any positive integers $c$ and $d$ such that $\frac{d}{c}>\frac{37}{228}$, the divisor $-cK_{X}+dH$ is ample.
In particular, the divisor $-4K_{X}+H$ is ample.
Thus, the divisor 
$$
-F_{a}+aE_{1}+E_{2}=(a+1)(-4K_{X}+H)
$$ 
is ample, for every $a\geq 0$. 
By Proposition \ref{10pts-ext2-main} and Proposition \ref{10pts-ext1-main}, a direct computation yields that the sequence \eqref{twisted-10pts-EC} satisfies the conditions in Set-up \ref{more-ex-setup}.
Let $\mathscr{C}_{X}$  be the left orthogonal complement of \eqref{twisted-10pts-EC}.
Then, there is an equivalence
$
-\otimes \CO_{X}(K_{X}): \mathscr{C}_{X}\rightarrow \RA_{X}^{(a)} \otimes \CO_{X}(-aE_{1}-E_{2}).
$
By Proposition \ref{uniform-example}, there exists an object 
$\mathcal{J}^{\prime}  \in ^{(\tau_{X})_{\mathscr{C}_{X}}^{-}}\mathscr{C}_{X}^{\heartsuit}$.
As a result, it gives an explicit object in the heart of the induced bounded $t$-structure on $\RA_{X}^{(a)}$ in Theorem \ref{bd-t-str-ourphant}.
\end{ex}

%=====================================================================
 
\section{Final remarks}\label{Final-remarks}

In this section, we propose possible ways to construct further new phantom categories via mutations and Krah's approach.

\subsection{One more phantom on ten-point blow-up}

Let $X$ be the blow-up of $\PB^{2}$ at $10$ points in general position.
Then, $X$ can be viewed as the blow-up of $S$ at $7$ points, where $S$ is the blow-up of $\PB^{2}$ at three points in general position.
In \cite{KN98}, Karpov--Nogin showed that $\DC(S)$ admits a three-block decomposition
$$
\DC(S)=\langle \CO_{S}, \CO_{S}(h_{1}), \CO_{S}(h_{2}), \CO_{S}(h_{3}), \CO_{S}(H),\CO_{S}(H^{\prime}) \rangle,
$$
where $H^{\prime}:=2H-E_{1}-E_{2}-E_{3}$ and $h_{i}:=H-E_{i}$, $i=1,2,3$.
Hence, by Orlov's blow-up formula and mutations, we obtain a full exceptional collection of line bundles on $X$
\begin{equation}\label{one-moreFEC-10pts} 
\{\CO_{X}, \begin{smallmatrix} \CO_{X}(E_{4}) \\ \cdots \\ \CO_{X}(E_{10}) \end{smallmatrix}, \begin{smallmatrix} \CO_{X}(H-E_{1}) \\ \CO_{X}(H-E_{3}) \\ \CO_{X}(H-E_{2}) \end{smallmatrix}, \CO_{X}(H),\CO_{X}(2H-E_{1}-E_{2}-E_{3})\}.
\end{equation}
By construction, both the full exceptional collections \eqref{one-moreFEC-10pts} and \eqref{standard-FEC-linebund} are mutation-equivalent to \eqref{standard-FEC} for $n=10$,
so \eqref{one-moreFEC-10pts} and \eqref{standard-FEC-linebund} are mutation-equivalent for $n=10$.
Now applying the involution \eqref{10pts-involution} to \eqref{one-moreFEC-10pts}, we get some new divisors:
$$
\begin{array}{r@{\hspace{3pt}}lr@{\hspace{3pt}}l}
G_{i}&:=4K_{X}-H+E_{i}, & M_{1}&:=6K_{X}-H, \\
M_{2}&:=6K_{X}-2H+E_{1}+E_{2}+E_{3},
\end{array}
$$
where $1\leq i\leq 3$. 
Likewise, using Dumnicki's diagram-cutting method, we can obtain the following:

\begin{thm}
Let $X$ be the blow-up of $\PB^{2}$ at $10$ points in general position.
Then the sequence 
$$
\{\CO_{X},\CO_{X}(B_{4}),\cdots,\CO_{X}(B_{10}),\CO_{X}(G_{1}),\CO_{X}(G_{2}),\CO_{X}(G_{3}),\CO_{X}(M_{1}),\CO_{X}(M_{2})\}
$$
is a non-full exceptional collection of line bundles of maximal length.
In particular, its right orthogonal complement $\RT_{X}$ is a universal phantom subcategory.
Moreover, the second Hochschild cohomology $\HH^{2}(\RT_{X})=\CN^{12}$ and the third Hochschild cohomology $\HH^{3}(\RT_{X})=\CN^{194}$.
\end{thm}

\begin{rem}
Similar to \cite{Kra24,KKL+26,MXY25}, this result can also be obtained by using the SHGH conjecture, as established in \cite{DJ07}. 
Moreover, likewise to Theorem \ref{bd-t-str-ourphant}, we can show that the phantom category $\RT_{X}$ has a bounded $t$-structure.
\end{rem}

\subsection{Countably many phantoms on eleven-point blow-up}
Let $X$ be the blow-up of $\mathbb{P}^{2}$ at $11$ points in general position. 
Then, for each integer $b\geq 0$, by Remark \ref{Ori-FEC-mut-equ-standard}, the full exceptional collection \eqref{Original-FEC} is mutation-equivalent to \eqref{standard-FEC-linebund} for $n=11$.
Similarly, applying the involution 
\begin{equation}\label{involution-Pic-11pts}
\begin{array}{cccl}
  \iota: & \mathrm{Pic}(X) & \longrightarrow  & \mathrm{Pic}(X)  \\
&D&\longmapsto &
-D-(D\cdot K_{X})K_{X},
\end{array}
\end{equation}
to \eqref{Original-FEC} for $n=11$, we get the divisors 
$$
\begin{array}{r@{\hspace{3pt}}lr@{\hspace{3pt}}l}
S:=& K_{X}-H+E_{1}+E_{2}, 
& T_j:=& K_{X}-E_{j}, \\
R:=&2K_{X}-H+E_{1}, &
N_{b}:=&2(b+1)K_{X}-(b+1)H+bE_{1}+E_{2},
\end{array}
$$
where $3\leq j\leq 11$.
Using Dumnicki's diagram-cutting method, in the forthcoming paper \cite{MXY26}, we will prove the following:

\begin{thm}
Let $X$ be the blow-up of $\mathbb{P}^{2}$ at $11$ points in general position.
Then, for every integer $b\geq 0$, the sequence
$$
\{\CO_{X},\CO_{X}(S),\CO_{X}(T_{3}),\cdots,\CO_{X}(T_{11}),\CO_{X}(R),\CO_{X}(N_{b}),\CO_{X}(N_{b+1})\}
$$
is a non-full exceptional collection of line bundles of maximal length. 
In particular, its right orthogonal complement $\RB_{X}^{(b)}$ is a universal phantom subcategory.
Moreover, the phantom category $\RB_{X}^{(b)}$ has a bounded $t$-structure.
\end{thm}

\subsection{Countably many phantoms on nine-point blow-up}

Let $X$ be the blow-up of the second Hirzebruch surface $\mathbf{F}_{2}$ at $9$ points in general position.
Following the notation in \cite[Section 4]{KKL+26}, we denote $C$ the pullback to $X$ of the negative section of $\mathbf{F}_{2}$, $F$ the pullback of the fiber of $\mathbf{F}_{2}$, and $E_{1}, \cdots, E_{9}$ the exceptional divisors. Then, the canonical divisor 
$$
K_{X}=-2C-4F+\sum_{i=1}^{9} E_{i}
$$ and the intersection numbers
$$ 
C^{2}=-2,\, C\ldotp F=1,\, F^{2}=0, 
$$
$$
C\ldotp E_{i}=0,\, F\ldotp E_{i}=0, \,E_{i}^{2}=-1,\, \textrm{ and } E_{i}\ldotp E_{j}=0\; (i\neq j).
$$
Then, by Orlov's blow-up formula,
for every integer $c\geq 0$, the sequence of line bundles
\begin{equation}\label{F2-family-FEC}
\{\mathcal{O},\mathcal{O}(E_{1}),\mathcal{O}(E_{2}),\cdots,\mathcal{O}(E_{9}),\mathcal{O}(F),\mathcal{O}(C+cF),\mathcal{O}(C+(c+1)F)\}  \end{equation}
is a full exceptional collection.
In \cite{KKL+26}, Kemboi et al.  considered the involution
\begin{equation}\label{involution-Pic-F2}
\begin{array}{cccl}
  \iota: & \mathrm{Pic}(X) & \longrightarrow  & \mathrm{Pic}(X)  \\
&D&\longmapsto &
-D-2(D\cdot K_{X})K_{X}.
\end{array}
\end{equation}
Applying the involution \eqref{involution-Pic-F2} to \eqref{F2-family-FEC},
there is a numerically exceptional collection
\begin{align}\label{F2-family-EC}
& \{\CO_{X},\CO_{X}(2K_{X}-E_{1}),\cdots,\CO_{X}(2K_{X}-E_{9}),\CO_{X}(4K_{X}-F), \nonumber\\
&\;\;\;\;\; \CO_{X}(c(4K_{X}-F)-C),\CO_{X}((c+1)(4K_{X}-F)-C)\}. 
\end{align}
With computer-assisted computations in Macaulay2 as explained in \cite[Appendix B]{KKL+26}, Kemboi et al. proved in \cite[Theorem 1.2]{KKL+26} that for 
$c=2$, the sequence \eqref{F2-family-EC} is a non-full exceptional collection.
Its right orthogonal complement gives the first example of universal phantom subcategory on the blow-up of $\mathbf{F}_{2}$ at $9$ points in general position.
It is important to note that for different integers $c$, the full exceptional collections \eqref{F2-family-FEC} are mutation-equivalent to each other.
In the forthcoming paper \cite{MXY26},
applying Dumnicki's diagram-cutting method without relying on Macaulay2, we will prove the following:

\begin{thm}
For every integer $c\geq 2$, the sequence \eqref{F2-family-EC} is a non-full exceptional collection. In particular, its right orthogonal complement $\mathscr{C}_{X}^{(c)}$ is a universal phantom subcategory. Moreover, the phantom category $\mathscr{C}_{X}^{(c)}$ has a bounded $t$-structure.
\end{thm}

On the other hand, since $H^{0}(X, \CO_{X}(C))\neq 0$, 
the sequence \eqref{F2-family-EC} fails to be exceptional when $c=0,1$. 
Consequently, there exist two mutation-equivalent full exceptional collections of line bundles on $X$, after applying the involution \eqref{involution-Pic-F2}, one becomes a non-full exceptional collection, but the other does not become an exceptional collection.
This phenomenon demonstrates that mutation-equivalence alone does not determine the behavior of a full exceptional collection under the involution \eqref{involution-Pic-F2}.
In contrast, for the blow-ups of $\PB^{2}$ at $n$ points in general position ($n=10$ or $n=11$), it is currently unknown whether an analogous phenomenon can occur.
Motivated by this observation, we formulate the following: 

\begin{quest}
Let $X$ be the blow-up of $\PB^{2}$ at $n$ points in general position, where $n=10$ or $n=11$.
Given any two mutation-equivalent full exceptional collections of line bundles on $X$,
if after applying the involution \eqref{involution-Pic} or \eqref{involution-Pic-11pts}, one becomes a non-full exceptional collection, does the same hold for the other?   
\end{quest} 

%=============================================================================================================================

\appendix 

\section{Proof of Proposition \ref{10pts-NoSect-case2}}\label{technique-prop2}

This appendix is to present the proof of Proposition \ref{10pts-NoSect-case2}.
The strategy of the proof is the same as that of Proposition \ref{10pts-NoSect-case1}.
To show Proposition \ref{10pts-NoSect-case2}, we use Proposition \ref{Dumnicki-method} to prove that 
the following linear systems 
$$
\mathscr{L}_{a}^{0}:=\mathcal{L}(13a+9,5a+2,(4a+3)^{9})
\textrm{ and }
\mathscr{L}_{a}^{1}:=\mathcal{L}(13a+16,5a+5,4a+6,(4a+5)^{8})
$$
are empty, for every integer $a\geq 0$.
To deal with the two systems simultaneously, 
we set $\sigma\in\{0,1\}$ and fix some notation:
$d=13a+9+7\sigma$,
$M=5a+2+3\sigma$,
$m=4a+3+2\sigma$
and
$q=4a+3+3\sigma$.
Then, we have $\mathscr{L}_a^{\sigma}=\mathcal{L}(d,M,q,m^{8})$.  

{\bf Step 1: Construct an affine partition.} 
For a small $0<\varepsilon<1$, we  define $9$ affine functions as follows:
\begin{align*}
f_{1}(x,y)&=-x-y+M-1+\varepsilon,&
f_{2}(x,y)&=x-d+q-1+\varepsilon,\\
f_{3}(x,y)&=y-d+m-1+\varepsilon,&
f_{4}(x,y)&=x-y-5a-4-2\sigma+\varepsilon,\\
f_{5}(x,y)&=-x+y-5a-4-3\sigma+\varepsilon,&
f_{6}(x,y)&=3x-y-15a-10-7\sigma+\varepsilon,\\
f_{7}(x,y)&=-3x+y+3a+1+\sigma
+\varepsilon,&
f_{8}(x,y)&=-x+y-a-1-\sigma+\varepsilon,\\
f_{9}(x,y)&=x+3y-15a-7-9\sigma+\varepsilon. &&
\end{align*}
Since $0<\varepsilon<1$, so no lattice point lies on one of the separating affine lines.
We denote $R_{0}:=\Delta_{d}$ and $R_{i}=R_{i-1}\cap\{(x,y)\mid f_{i}(x,y)<0\}$.    
Inductively, the diagrams $\mathrm{D}_{1},\cdots,\mathrm{D}_{10}$ are defined by
\begin{equation}\label{partition-def2}
\mathrm{D}_{i}:=R_{i-1}\cap \{(x,y)\mid f_{i}(x,y)>0\}, \; 1\leq i\leq 9,
\end{equation}
and $\mathrm{D}_{10}:=R_{9}$.  
Therefore, we get an affine partition
$\Delta_d=\mathrm{D}_1\cup\cdots\cup \mathrm{D}_{10}$; see for example, Figure \ref{cutting-diag2-figure} for $a=2$, $\sigma=0$ and $\varepsilon=0.1$.

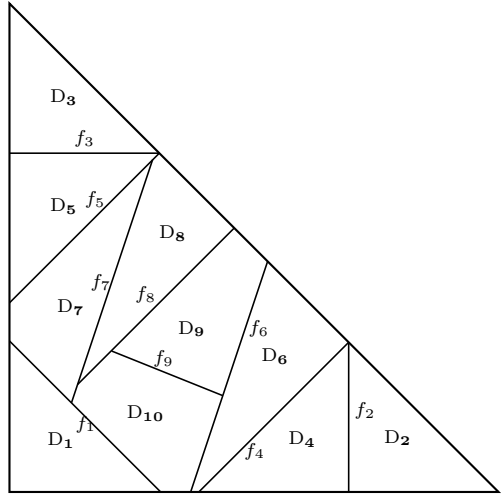
\begin{figure}[H]
\centering
\begin{tikzpicture}[scale=0.18, font=\sffamily]
\draw[line width=0.8pt] (0, 0) -- (35.9, 0) -- (0, 35.9) -- cycle;
% f3: y = 24.9 
\draw[semithick] (0, 24.9) -- (11.0, 24.9);
% f5: y = x + 13.9 
\draw[semithick] (0, 13.9) -- (11.0, 24.9);
% f2: x = 24.9
\draw[semithick] (24.9, 0) -- (24.9, 11.0);
% f4: y = x - 13.9 
\draw[semithick] (13.9, 0) -- (24.9, 11.0);
% f1: y = -x + 11.1 
\draw[semithick] (0, 11.1) -- (11.1, 0);
% f7: y = 3x - 7.1
\draw[semithick] (4.55, 6.55) -- (10.51, 24.43);
% f8: y = x + 2.9
\draw[semithick] (5.0, 7.9) -- (16.5, 19.4);
% f6: y = 3x - 39.9
\draw[semithick] (13.3, 0) -- (18.95, 16.95);
% f9: y = -x/3 + 12.3
\draw[semithick] (7.5, 10.36) -- (15.66, 7.08);
%=================================
\node at (4, 29) {\tiny $\mathrm{D}_{\mathbf{3}}$};
\node at (5.5, 26) {\tiny $f_{3}$};
\node at (4, 21) {\tiny $\mathrm{D}_{\mathbf{5}}$};
\node at (6.3, 21.5) {\tiny $f_{5}$};
\node at (4.5, 13.6) {\tiny $\mathrm{D}_{\mathbf{7}}$};
\node at (6.7, 15.4) {\tiny $f_{7}$};
\node at (3.8, 3.8) {\tiny $\mathrm{D}_{\mathbf{1}}$};
\node at (5.6, 5) {\tiny $f_{1}$};
\node at (12, 19) {\tiny $\mathrm{D}_{\mathbf{8}}$};
\node at (10, 14.5) {\tiny $f_{8}$};
\node at (13.4, 12) {\tiny $\mathrm{D}_{\mathbf{9}}$};
\node at (11.3, 9.8) {\tiny $f_{9}$};
\node at (10, 5.8) {\tiny $\mathrm{D}_{\mathbf{10}}$};
\node at (19.5, 10) {\tiny$\mathrm{D}_{\mathbf{6}}$};
\node at (18.3, 12) {\tiny $f_{6}$};
\node at (21.5, 3.9) {\tiny $\mathrm{D}_{\mathbf{4}}$};
\node at (18, 3) {\tiny $f_{4}$};
\node at (28.5, 4) {\tiny $\mathrm{D}_{\mathbf{2}}$};
\node at (26.1, 6) {\tiny $f_{2}$};
\end{tikzpicture}
\caption{An affine partition of $\Delta_{35}$}
\label{cutting-diag2-figure}
\end{figure}

{\bf Step 2: 
Verify the conditions of Proposition \ref{Dumnicki-method}.}
We discuss the diagrams $\mathrm{D}_{1},\cdots, \mathrm{D}_{10}$ case-by-case.

To begin with, it follows directly from \eqref{partition-def2} that for the diagrams $\mathrm{D}_{1}$, $\mathrm{D}_{2}$, $\mathrm{D}_{3}$, $\mathrm{D}_{4}$ and $\mathrm{D}_{5}$, we have:
\begin{enumerate}
\item[(1)] For the diagram $\mathrm{D}_{1}$, we have
$
\displaystyle 
\mathrm{D}_{1} =\bigcup_{r=1}^{M} \{(r-1,y)\mid 0\leq y\leq M-r\}.
$
Thus, the cardinality $\# \mathrm{D}_{1}=\frac{M(M+1)}{2}$. 
We take $M$ vertical lines $\ell_{1,k}:=\{x=M-k\}$, 
$1\leq k\leq M$.
Then, the cardinality $\#(\mathrm{D}_{1}\cap \ell_{1,k})=k$, $1\leq k\leq M$.

\item[(2)] For the diagram $\mathrm{D}_{2}$,
we have
$\displaystyle 
\mathrm{D}_{2} =\bigcup_{r=1}^{q}
 \{(d-q+r,y) \mid 0\leq y\leq q-r\}.
$
Thus, the cardinality $\# \mathrm{D}_{2}=\frac{q(q+1)}{2}$.
We take $q$ vertical lines $\ell_{2,k}:=\{x=d-k+1\}$, 
$1\leq k\leq q$.
Then the cardinality $\#(\mathrm{D}_{2}\cap \ell_{2,k})=k$, $1\leq k\leq q$.

\item[(3)] For the diagram $\mathrm{D}_{3}$, 
we have
$\displaystyle 
\mathrm{D}_{3} =\bigcup_{r=1}^{m}
 \{(x,d-m+r)\mid 0\leq x\leq m-r\}.
$
Thus, the cardinality $\# \mathrm{D}_{3}=\frac{m(m+1)}{2}$.
We take $m$ horizontal lines $\ell_{3,k}:=\{y=d-k+1\}$, 
$1\leq k\leq m$.
Then, the cardinality $\#(\mathrm{D}_{3}\cap \ell_{3,k})=k$, $1\leq k\leq m$.

\item[(4)] For the diagram $\mathrm{D}_{4}$,
we have 
$\displaystyle 
\mathrm{D}_{4}=\bigcup_{r=1}^{m}
 \{(5a+3+2\sigma+r,y)\mid 0\leq y\leq r-1\}.
$
Thus, the cardinality $\# \mathrm{D}_{4}=\frac{m(m+1)}{2}$.
We take $m$ vertical lines $\ell_{4,k}:=\{x=5a+3+2\sigma+k\}$, 
$1\leq k\leq m$.
Then, the cardinality $\#(\mathrm{D}_{4}\cap \ell_{4,k})=k$, $1\leq k\leq m$.

\item[(5)] For the diagram $\mathrm{D}_{5}$, 
we have 
$\displaystyle 
\mathrm{D}_{5} =\bigcup_{r=1}^{m}
 \{(r-1,y) \mid 5a+3+3\sigma+r\leq y\leq 9a+6+5\sigma\}.
$
Thus, the cardinality $\# \mathrm{D}_{5}=\frac{m(m+1)}{2}$.
We take $m$ horizontal lines $\ell_{5,k}:=\{x=4a+3+2\sigma-k\}$, 
$1\leq k\leq m$.
Then, the cardinality $\#(\mathrm{D}_{5}\cap \ell_{5,k})=k$, $1\leq k\leq m$.
\end{enumerate}

Additionally, the discussions are the same for the diagrams $\mathrm{D}_{6}$, $\mathrm{D}_{7}$ and $\mathrm{D}_{8}$:
\begin{enumerate}
\item[(i)]For the diagram $\mathrm{D}_{6}$, the conditions of not belonging to $\mathrm{D}_{4}$, belonging to $\mathrm{D}_{6}$, and lying in $\Delta_{d}$ yield the condition:
$
x-5a-3-2\sigma\leq y\leq \min\{3x-15a-10-7\sigma,d-x\}.
$
Thus, we have 
$$ \displaystyle 
\mathrm{D}_{6}=\bigcup_{r=1}^{m}
 \{(5a+3+2\sigma+r,y) \mid r\leq y\leq\min\{3r-1-\sigma,8a+6+5\sigma-r\}\}.
$$
If $\sigma=0$, we take $m$ vertical lines as follows: 
(i) if $k$ is even, then $\ell_{6,k}:=\{x=5a+3+\frac{k}{2}\}$;
(ii) if $k$ is odd, then $\ell_{6,k}:=\{x=9a+6-\frac{k-1}{2}\}$.
If $\sigma=1$, 
we take $m$ vertical lines as follows: 
(i) if $k$ is even, then $\ell_{6,k}:=\{x=9a+11-\frac{k}{2}\}$;
(ii) if $k$ is odd, then $\ell_{6,k}:=\{x=5a+5+\frac{k+1}{2}\}$.
Thus, the cardinality $\#(\mathrm{D}_{6}\cap \ell_{6,k})=k$, $1\leq k\leq m$, and then $\# \mathrm{D}_{6}=\frac{m(m+1)}{2}$.

\item[(ii)] For the diagram $\mathrm{D}_{7}$, after simplifying the preceding inequalities, one obtains the constraint conditions:
$ 0\leq x\leq m-1$ and $\max\{M-x,3x-3a-1-\sigma\}\leq y\leq x+5a+3+3\sigma$.
Thus, we get
$$\displaystyle 
\mathrm{D}_{7}=\bigcup_{r=1}^{m}
  \{(r-1,y) \mid \max\{M-r+1,3r-3a-4-\sigma\}\leq y \leq 5t+2+3\sigma+r
 \}.
$$
We take $m$ vertical lines as follows: (i) if $k$ is even, then $\ell_{7,k}=\{x=\frac{k}{2}\}$; (ii) if $k$ is odd, then $\ell_{7,k}=\{x=m-\frac{k+1}{2}\}$.
Thus, the cardinality $\#(\mathrm{D}_{7}\cap \ell_{7,k})=k$, $1\leq k\leq m$, and then $\# \mathrm{D}_{7}=\frac{m(m+1)}{2}$.

\item[(iii)] For the diagram $\mathrm{D}_8$, the constraint conditions are
$2a+2+\sigma\leq x \leq 2a+1+m+\sigma$ and $x+a+1+\sigma\leq y \leq \min\{3x-3a-2-\sigma,d-x\}$.
Hence, we have
$$\displaystyle 
\mathrm{D}_{8}=\bigcup_{r=1}^{m}
 \left\{(2a+1+\sigma+r,y) \mid
 l_{8}(r)\leq y\leq
 u_{8}(r)
 \right\},
$$
where $l_{8}(r):=3a+2+2\sigma+r$ and $u_{8}(r):=\min\{3a+1+3r+2\sigma,11a+8+6\sigma-r\}$.
We take $m$ vertical lines as follows: (i) if $k$ is even, then $\ell_{8,k}=\{x=2a+1+\sigma+\frac{k}{2}\}$; (ii) if $k$ is odd, then $\ell_{8,k}=\{x=4a+2\sigma+3-\frac{k-1}{2}\}$.
Thus, the cardinality $\#(\mathrm{D}_{8}\cap \ell_{8,k})=k$, $1\leq k\leq m$, and $\# \mathrm{D}_{8}=\frac{m(m+1)}{2}$.
\end{enumerate}

Finally, the discussions are also the same for the last two diagrams:
\begin{enumerate}
\item[(1)]For the diagram $\mathrm{D}_{9}$,
set $x=3a+1+2\sigma+r$. The new inequality $x+3y-15a-7-9\sigma+\epsilon>0$, together
with the complements of the $\mathrm{D}_{6}$- and $\mathrm{D}_{8}$-inequalities and the boundary
of $\Delta_{d}$, gives the condition: 
$l_{9}(r) \leq y \leq u_{9}(r)$,
where 
$l_{9}(r):=\max \{
 4a+2+2\sigma-\left\lfloor\frac{r-\sigma}{3}\right\rfloor,
 3r-6a-6-\sigma \}$ and 
 $u_{9}(r):=\min\{4a+1+3\sigma+r,\,10a+8+5\sigma-r\}$.
Thus, we have
$$\displaystyle 
\mathrm{D}_{9}=\bigcup_{r=1}^{m}
 \{(3a+1+2\sigma+r,y)\mid l_9(r)\leq y\leq u_{9}(r)\}.
$$
We take $m$ vertical lines as follows:
\begin{enumerate}
\item[(i)] If $\sigma=0$, then there are four cases:
(1) if $k\equiv 0\; ({\rm mod}\; 4)$, then $\ell_{9,k}=\{x=3a+1+\frac{3}{4}k\}$;
(2) if $k\equiv 1\; ({\rm mod}\; 4)$, then $\ell_{9,k}=\{x=3a+1+\frac{3k+1}{4}\}$;
(3) if $k\equiv 2\;({\rm mod}\; 4)$, then $\ell_{9,k}=\{x=3a+1+\frac{3k+2}{4}\}$;
(4) if $k\equiv 3\; ({\rm mod}\; 4)$, then $\ell_{9,k}=\{x=3a+1+m-\frac{k-3}{4}\}$.
Thus, the cardinality$\#(\mathrm{D}_{9}\cap\ell_{9,k})=k$, $1\leq k\leq m$ and $\#\mathrm{D}_{9}=\frac{m(m+1)}{2}-1$.
\item[(ii)] If $\sigma=1$, then there exist four cases:
(1) if $k\equiv 0\; ({\rm mod}\; 4)$, then $\ell_{9,k}=\{x=3a+3+\frac{3}{4}k\}$;
(2) if $k\equiv 1\; ({\rm mod}\; 4)$, then $\ell_{9,k}=\{x=3a+3+m-\frac{k-1}{4}\}$;
(3) if $k\equiv 2\; ({\rm mod}\; 4)$, then $\ell_{9,k}=\{x=3a+3+\frac{3k-2}{4}\}$;
(4) if $k\equiv 3\; ({\rm mod}\; 4)$, then $\ell_{9,k}=\{x=3a+3+\frac{3k-1}{4}\}$.
Thus, the cardinality$\#(\mathrm{D}_{9}\cap\ell_{9,k})=k$, $1\leq k\leq m$ and $\#\mathrm{D}_{9}=\frac{m(m+1)}{2}-1$.
\end{enumerate}

\item[(2)]For the diagram $\mathrm{D}_{10}$, on the horizontal line
$y=r-\sigma$, the inequalities defining the residual piece $\mathrm{D}_{10}$
reduce to $ l_{10}(r) \leq x \leq u_{10}(r)$,
where 
$l_{10}(r):=\max\{q+1-\sigma-r,\,r-\sigma-a\}$ 
and
$u_{10}(r):=\min\left\{
 q+\left\lfloor\frac{r-\sigma}{3}\right\rfloor,
 15a+6+\sigma-3r \right\}$.
Thus, we have
$$ \displaystyle 
\mathrm{D}_{10}
=\bigcup_{r=1}^{m}
 \{(x,r-\sigma)\mid l_{10}(r)\leq x\le u_{10}(r)\}.
$$
We take $m$ horizontal lines as follows:
\begin{enumerate}
\item[(i)] If $\sigma=0$, then there are four cases:
(1) if $k\equiv 0\; ({\rm mod}\; 4)$, then $\ell_{10,k}=\{y=\frac{3}{4}k-1\}$;
(2) 
if $k=1$, then $\ell_{10,k}=\{y=m-1\}$;
if $k\equiv 1\; ({\rm mod}\; 4)$ ($k\geq 5$), then $\ell_{10,k}=\{y=\frac{3k-11}{4}\}$;
(3) if $k\equiv 2\;({\rm mod}\; 4)$, then $\ell_{10,k}=\{y=\frac{3(k-2)}{4}\}$;
(4) if $k\equiv 3\; ({\rm mod}\; 4)$, then $\ell_{10,k}=\{y=4a+1-\frac{k-3}{4}\}$.
Thus, 
in case (1), (3) and (4), we have $\#(\mathrm{D}_{10}\cap\ell_{10,k})=k$
for the corresponding value of $k$,
whereas in case (2) one has
the cardinality $\#(\mathrm{D}_{10}\cap\ell_{10,k})<k$.
It follows that $\#\mathrm{D}_{10}=\frac{m^{2}+1}{2}$.
\item[(ii)] If $\sigma=1$, then there exist four cases:
(1) if $k\equiv 0\; ({\rm mod}\; 4)$, then $\ell_{10,k}=\{y=\frac{3}{4}k-1\}$;
(2) if $k\equiv 1\; ({\rm mod}\; 4)$, then $\ell_{10,k}=\{y=m-1-\frac{k-1}{4}\}$;
(3) if $k\equiv 2\; ({\rm mod}\; 4)$, then $\ell_{10,k}=\{y=\frac{3k-2}{4}\}$;
(4) if $k\equiv 3\; ({\rm mod}\; 4)$, then $\ell_{10,k}=\{y=\frac{3k-9}{4}\}$.
Thus, 
in case (1), (2) and (3), we have $\#(\mathrm{D}_{10}\cap\ell_{10,k})=k$
for the corresponding value of $k$,
whereas in case (4) one has
the cardinality $\#(\mathrm{D}_{10}\cap\ell_{10,k})<k$.
By direct computations, we conclude that $\#\mathrm{D}_{10}=\frac{m^{2}+1}{2}$.
\end{enumerate}
\end{enumerate}
As a result, by Proposition \ref{Dumnicki-method}, we get $H^{0}(X,\CO_{X}(D))=0$.  
This completes the proof of Proposition \ref{10pts-NoSect-case2}.
 
%=====================================================================


\begin{thebibliography}{100}

\bibitem{AO13}
V. Alexeev, D. Orlov,
{\it Derived categories of Burniat surfaces and exceptional collections}, 
Math. Ann. {\bf 357} (2013), 743--759.

\bibitem{BBC+12}
T. Bauer, C. Bocci, S. Cooper, S. Di Rocco, M. Dumnicki, B. Harbourne, and K. Jabbusch, A.L. Knutsen, A. K\"{u}ronya, R. Miranda, J. Ro\'{e}, H. Schenck, T. Szemberg, and Z. Teitler, 
{\it Recent Developments and Open Problems in Linear Series}, Contributions to algebraic geometry, 93--140, EMS Ser. Congr. Rep., Eur. Math. Soc., Z\"{u}rich, 2012.

\bibitem{BGvBKS15}
C. B\"{o}hning, H.-C. Graf von Bothmer, L. Katzarkov, P. Sosna,
{\it Determinantal Barlow surfaces and phantom categories}, 
J. Eur. Math. Soc. {\bf 17} (2015), 1569--1592.

\bibitem{BGvBS13}
C. B\"{o}hning, H.-C. Graf von Bothmer, P. Sosna,
{\it On the derived category of the classical Godeaux surface}, 
Adv. Math. {\bf 243} (2014), 203--231.
 
\bibitem{Bon90}
A. Bondal, 
{\it Representations of associative algebras and coherent sheaves}, 
Math. USSR-Izv. {\bf 34} (1990), 23--42.

\bibitem{BK25}
L. Borisov and K. Kemboi, 
{\it Non-existence of phantoms on some non-generic blowups of the projective plane}, 
Proc. Amer. Math. Soc. {\bf 153} (2025), 963--968.
 
\bibitem{CL18}
Y. Cho, Y. Lee,
{\it Exceptional collections on Dolgachev surfaces associated with degenerations},
Adv. Math. {\bf 324} (2018), 394--436.

\bibitem{CM11}
C. Ciliberto, R. Miranda, 
{\it Homogeneous interpolation on ten points}, 
J. Algebraic Geom. {\bf 20} (2011), 685--726.

\bibitem{Dum07}
M. Dumnicki, 
{\it Cutting diagram method for systems of plane curves with base points}, 
Ann. Polon. Math. {\bf 90} (2007), 131--143.

\bibitem{DJ07}
M. Dumnicki, W. Jarnicki,
{\it New effective bounds on the dimension of a linear system in $\mathbb{P}^{2}$},
J. Symb. Comput. {\bf 42} (2007), 621--635.

\bibitem{GKMS15}
S. Galkin, L. Katzarkov, A. Mellit, E. Shinder,
{\it Derived categories of Keum's fake projective planes},
Adv. Math. {\bf 278} (2015), 238--253.

\bibitem{GS13}
S. Galkin, E. Shinder,
{\it Exceptional collections of line bundles on the Beauville surface},
Adv. Math. {\bf 244} (2013), 1033--1050.
 
\bibitem{GO13} 
S. Gorchinskiy,  D. Orlov, 
{\it Geometric phantom categories}, 
Publ. Math. Inst. Hautes \'{E}tudes Sci. {\bf 117} (2013), 329--349.

\bibitem{KN98}
B.V. Karpov, D.Yu. Nogin,
{\it Three-block exceptional sets on del Pezzo surfaces}, 
Izv. Math. {\bf 62} (1998), 3--38.

\bibitem{KK23}
I. Karzhemanov, L. Katzarkov,
{\it Exceptional collections and phantoms of special Dolgachev surfaces},
\href{https://arxiv.org/abs/2310.13319v1}{arXiv:2310.13319v1}.

\bibitem{KKL+26}
K. Kemboi, D. Krashen, T. Liu, Y. Liu, E. Mackall, S. Makarova, A. Perry, A. Robotis, S. Venkatesh,
{\it A looming of phantoms}, 
Adv. Math. {\bf 499} (2026), 111046.

\bibitem{Kra24}
J. Krah,
{\it A phantom on a rational surface}, 
Invent. Math. {\bf 235} (2024), 1009--1018.

\bibitem{KO95}
S. Kuleshov, D. Orlov,
{\it Exceptional sheaves on del Pezzo surfaces}, 
Russian Acad. Sci. Izv. Math. {\bf 44} (1995), 479--513.
 
\bibitem{Kuz15}
A. Kuznetsov, 
{\it Height of exceptional collections and Hochschild cohomology of quasiphantom categories},
J. Reine Angew. Math. {\bf 708} (2015), 213--243.

\bibitem{KLP26}
A. Kuznetsov, S. Liu, A. Perry,
{\it Inducing $t$-structures on semiorthogonal components},
\href{https://arxiv.org/abs/2606.26193}{arXiv:2606.26193v1}.

\bibitem{Liu25}
Y. Liu,
{\it Geometric phantom categories do not admit Noetherian $t$-structures},
\href{https://arxiv.org/abs/2503.02052}{arXiv:2503.02052v1}.

\bibitem{MXY25}
S. Ma, Y. Xiong, S. Yang,
{\it A new phantom on a rational surface}, 
\href{https://arxiv.org/abs/2511.07114v2}{arXiv:2511.07114v2}. 

\bibitem{MXY26}
S. Ma, Y. Xiong, S. Yang,
{\it Echoes of phantoms on rational surfaces II},
In preparation.

\bibitem{Ma68}
Y. Manin,
{\it Correspondences, motifs and monoidal transformations},
Math. USSR, Izv. {\bf 21} (1983), 307--340.

\bibitem{Mat25}
A. Mattoo,
{\it Objects of a phantom on a rational surface},
\href{https://arxiv.org/abs/2510.26107v1}{arXiv:2510.26107v1}.

\bibitem{Orl93}
D. Orlov, 
{\it Projective bundles, monoidal transformations, and derived categories of coherent sheaves},
Russian Acad. Sci. Izv. Math. {\bf 41} (1993), 133--141.
 
\bibitem{Pir23}
D. Pirozhkov,
{\it Admissible subcategories of del Pezzo surfaces}, 
Adv. Math. {\bf 424} (2023), 109046.
 
\bibitem{Sos20}
P. Sosna, 
{\it Some remarks on phantom categories and motives}, 
Bull. Belg. Math. Soc. Simon Stevin {\bf 27} (2020), 337--352.
 
\end{thebibliography}
\end{document}